\title{Calabi-Yau Caps, Uniruled Caps and Symplectic Fillings }

\documentclass[11pt]{article}

\usepackage{amsfonts}
\usepackage{amsthm}
\usepackage{xypic}
\usepackage{xcolor}
\usepackage{graphicx}
\usepackage{amsmath}
\usepackage{enumerate}
\usepackage[all]{xy}

\usepackage{comment}

\usepackage{amsmath,amsthm,amsfonts,amssymb,graphicx,psfrag}
\usepackage{color}

\newtheorem{thm}{Theorem}[section]

\newtheorem{prop}[thm]{Proposition}
\newtheorem{lemma}[thm]{Lemma}
\newtheorem{corr}[thm]{Corollary}
\newtheorem{conj}[thm]{Conjecture}

\theoremstyle{definition}
\newtheorem{rmk}[thm]{Remark}
\newtheorem{example}[thm]{Example}

\newtheorem{defn}[thm]{Definition}
\newtheorem{eg}[thm]{Example}

\begin{document}

\author{Tian-Jun Li, Cheuk Yu Mak  and Kouichi Yasui\thanks{
2010 Mathematics Subject Classification 57R17 (primary), 53D35 (secondary).
The first and second author are supported by NSF-grant DMS 1065927. The third author was partially supported by JSPS KAKENHI Grant Number 16K17593.}}

\AtEndDocument{\bigskip{\footnotesize%
  \textsc{School of Mathematics, University of Minnesota, Minneapolis, MN 55455} \par
  \textit{E-mail address}, Tian-Jun Li: \texttt{tjli@math.umn.edu} \par
  \addvspace{\medskipamount}
  \textsc{School of Mathematics, University of Minnesota, Minneapolis, MN 55455} \par
  \textit{E-mail address}, Cheuk Yu Mak: \texttt{makxx041@math.umn.edu} \par
\addvspace{\medskipamount}
  \textsc{Department of Mathematics, Graduate School of Science, Hiroshima University, 1-3-1 Kagamiyama, Higashi-Hiroshima, 739-8526, Japan} \par
  \textit{E-mail address}, Kouichi Yasui: \texttt{kyasui@hiroshima-u.ac.jp} \par

}}



\maketitle

\begin{abstract}
We introduce symplectic Calabi-Yau caps to obtain new obstructions to exact fillings.
In particular, they imply that any exact filling of the standard contact structure on the unit cotangent bundle of a hyperbolic surface has vanishing first Chern class
and has the same integral homology and intersection form as its disk cotangent bundle.
This gives evidence to a conjecture that all of its exact fillings are diffeomorphic to the disk cotangent bundle.
As a result, we also obtain the first infinite family of Stein fillable contact 3-manifolds with uniform bounds on the
Betti numbers of its exact fillings  but admitting minimal strong fillings of arbitrarily large $b_2$.

Moreover, we introduce the notion of symplectic uniruled/adjunction caps and uniruled/adjunction contact structures to present
a unified picture to the existing finiteness results
on the topological invariants of exact/strong fillings of a contact 3-manifold.
As a byproduct, we  find new classes of contact 3-manifolds with the finiteness properties and extend Wand's obstruction of
planar contact 3-manifolds to uniruled/adjunction contact structures with complexity zero.

\end{abstract}

\section{Introduction}

 Understanding symplectic fillings of a given contact 3-manifold $(Y,\xi)$ is a very active research area.
An ultimate goal is to classify all the Stein, exact or minimal strong symplectic fillings of a given contact manifold $(Y,\xi)$.
The first step towards this goal is to understand whether the given $(Y,\xi)$ has finitely many or infinitely many fillings.
Some families of contact 3-manifolds that admit finitely many Stein fillings are found (\cite{El91}, \cite{Mc90}, \cite{Lis08}, \cite{PlVHM10}, \cite{St13}, \cite{KaL}, \cite{GoLi14}, \cite{LiMa14}, etc).
For minimal strong fillings, Ohta, Ono and others have systematically investigated the links of  isolated singularities (\cite{OhOn99},  \cite{OhOn03}, \cite{OhOn05}, \cite{OhOn08},  \cite{BhOn} etc), and established uniqueness/finiteness/infiniteness for different classes of singularities.

Instead of classifying completely all the fillings, one can ask whether topological quantities for fillings are bounded.
It was conjectured by Stipsicz in \cite{Sti03} that all possible Euler characteristics and signatures of Stein fillings of a fixed $(Y,\xi)$ are bounded.
However, it was disproved by Baykur and Van Horn-Morris in \cite{BaVHM15}.
Based on the work of many people (\cite{OS04AMS}, \cite{AkEtMaSm08}, \cite{AY14}, \cite{BaVHM15}, \cite{BaVHM1212}, \cite{BaMoVHM14}, \cite{AkhOz2}, \cite{DKP13}, \cite{Y14}, etc), we now
know that many contact 3-manifolds have infinitely many Stein fillings up to
diffeomorphism.

Even though Stipsicz's conjecture is not true in general, it is important to know for what contact 3-manifolds the boundedness does hold.
The focus of the current paper is to address when the Betti numbers are bounded.

\begin{defn}
A contact 3-manifold $(Y,\xi)$ is of {\bf Stein} (resp. {\bf exact}, {\bf strong}) {\bf Betti finite type}
if there are only finitely many possible values of the tuple  $(b_1, b_2, b_3)$ of Betti numbers for all of its
Stein (resp. exact, minimal strong) fillings.
\end{defn}

We recall that Stein fillings are exact fillings, and exact fillings are minimal strong fillings.
Note also that this finiteness of Betti numbers guarantees the finiteness of $e$ and $\sigma$.
Planar contact 3-manifolds (i.e. contact 3-manifolds supported by open books of page genus zero) have  Stein Betti finite type (\cite{P12}, \cite{Ka13}), but there are many other contact manifolds having this property too.

To study this question, we introduce three kinds of caps, namely Calabi-Yau caps, uniruled caps and adjunction caps.
Calabi-Yau caps give surprising new restrictions to exact fillings and in particular, we apply it to study exact fillings of unit cotangent bundles.
Uniruled and adjunction caps unify several known finiteness results  into a single picture.
Along the way, we will give some alternative proofs for these known results and strengthen them.
All contact manifolds in this paper are assumed to be closed, three dimensional and have co-oriented contact structures.

\subsection{Calabi-Yau caps}

\begin{defn}\label{d:CY}
 A {\bf Calabi-Yau cap} of a contact 3-manifold $(Y,\xi)$ is a compact symplectic manifold $(P,\omega)$ which is a strong concave filling of $(Y,\xi)$ such that
 $c_1(P)$ is torsion.
\end{defn}

\begin{thm}\label{t:CY}
Suppose $(Y, \xi)$ admits a Calabi-Yau cap $(P,\omega_P)$.
Then $(Y, \xi)$ is of exact Betti finite type.

If, moreover, $(P,\omega_P)$ cannot be embedded in a uniruled manifold,
then all exact fillings of $(Y,\xi)$ have torsion first Chern class.

\end{thm}

It is worthy to point out that in many situations, there are simple (topological) obstructions for a Calabi-Yau cap $(P,\omega_P)$
to be embedded in a uniruled manifold.
On the other hand, the ingredients in the proof of Theorem \ref{t:CY} can be used to obtain the following surprising consequence.

\begin{thm}\label{t:lagrangian}
Let $Y$ be the unit cotangent bundle of a closed orientable surface $\Sigma_g$ of genus $g$, equipped with the standard contact structure $\xi_{std}$.
Then any exact filling of $(Y,\xi_{std})$ has the same integral homology and intersection form as $T^*\Sigma_g$ and has vanishing first Chern class.
\end{thm}

We remark that when $g=0,1$, $(Y,\xi_{std})$ has a unique exact filling up to symplectic deformation equivalence given by the disk cotangent bundle (See \cite{Hi00}, \cite{We10} and \cite{Sti02}).
However, there has been no good understandings of exact fillings for  $g>1$.
In contrast, no exact/Stein filling of $(Y,\xi_{std})$ that is not diffeomorphic to $T^*\Sigma_g$ has been found.
We heard from Chris Wendl the following conjecture.

\begin{conj}\label{cj:lagrangian}
The diffeomorphism types of Stein/exact fillings of $(Y,\xi_{std})$ is unique and given by the disk cotangent bundle for any $g$.
\end{conj}

Therefore, Theorem \ref{t:lagrangian} gives strong evidence to Conjecture \ref{cj:lagrangian}. After the first version of this paper  appeared
Sivek and Van Horn-Morris in  \cite{SiVHM15} were able to use the ideas of Calabi-Yau caps introduced here to obtain parts of Theorem \ref{t:lagrangian}
and derive  the beautiful result that the Stein version of Conjecture \ref{cj:lagrangian} is true up to s-cobordism relative to boundary.


As an immediate application of Theorem \ref{t:lagrangian}, we obtain Corollary \ref{c:infiniteFamily}.
This subtle difference between exact Betti finite type and strong Betti finite type is an interesting phenomenon that we do not find in the literature.
It also implies that Theorem \ref{t:CY} cannot be strengthened to strong Betti finite type.

\begin{corr}\label{c:infiniteFamily}
 There exist infinitely many Stein fillable contact 3-manifolds such that each of them is of exact Betti finite type but not of strong Betti finite type.
\end{corr}

\subsection{Uniruled/Adjunction caps}

For any symplectic cap (symplectic concave filling) $(P,\omega)$ of $(Y,\xi)$, there is Liouville vector field $V$ defined near $Y$ pointing inward along $Y$.
The induced one form $\alpha=\iota_V \omega$ is a contact one form on $\partial P$ such that $(\partial P, \ker (\alpha))$ is contactomorphic to $(Y,\xi)$.
For any choice of $V$, we call the induced one form $\alpha$ a {\bf Liouville one form}.
Given a Liouville one form $\alpha$, $[(\omega,\alpha)]$ is a relative cohomology class in $H^2(P ,\partial P ,\mathbb{R})$.

\begin{defn}\label{d:uniruled}
 A {\bf uniruled cap}  of a contact 3-manifold $(Y,\xi)$ is a symplectic concave filling $(P,\omega)$ of $(Y,\xi)$ such that $c_{1}(P)\cdot[(\omega,\alpha)]>0$
 for some Liouville one form $\alpha$.
\end{defn}

Since $[(\omega,\alpha)]$ is a relative class, $c_{1}(P)\cdot[(\omega,\alpha)]$ is well-defined and it is further explained in Subsection \ref{ss:RelativePairing}.
We call a contact 3-manifold admitting  a uniruled cap a {\bf uniruled contact manifold} (See Proposition \ref{prop:uniruled} for its relation with uniruled symplectic manifolds).
 A contact 3-manifold that is strong symplectic cobordant to a uniruled contact manifold is also uniruled (see Lemma \ref{stable under strong cobordism}).
 The class of uniruled contact manifolds is strictly larger than the planar class (see Lemma~\ref{rmk: uniruled-planar} and note that $(T^3,\xi_{std})$ is non-planar but admits a uniruled cap).

If $(Y,\xi)$ admits a uniruled cap, then we can derive  restrictions on fillings that are strictly stronger than when $(Y, \xi)$ admits a Calabi-Yau cap.

\begin{thm}\label{t:uniruled}
 Suppose a contact 3-manifold $(Y,\xi)$ admits a uniruled cap $(P,\omega_P)$.
 Then $(Y, \xi)$ is of strong Betti finite type.
\end{thm}

\begin{rmk}
We would like to point out that to the best of our knowledge, most previously known exact/strong Betti finite type contact 3-manifolds have a uniruled cap.
\end{rmk}

We also introduce another type of caps which we call {\bf adjunction caps}.
It is based on an observation that existence of a {\bf smoothly} embedded surface in a closed symplectic manifold with sufficiently large
self-intersection number relative to the genus  implies that the symplectic manifold is uniruled (Proposition \ref{prop:adjunction}).
Adjunction caps have similar properties as uniruled caps, including Theorem~\ref{t:uniruled}.
Similarly, every planar contact 3-manifold also has an adjunction cap (see \cite{Et04}, cf. Lemma \ref{rmk: uniruled-planar} )
and we call a contact manifold admitting an adjunction cap an {\bf adjunction contact manifold}.


As far as the examples we will see in this paper, all uniruled caps are adjunction and vice versa.
It is possible that these are in fact  the same notion.
For topologists, adjunction caps are often easier to find than uniruled caps.

\subsection{Genus invariants}

  We define  genus type invariants for uniruled/adjunction caps, which describe the range of possible genera of the base of uniruled symplectic manifolds that can be obtained from the caps.
  These invariants induce a complexity invariant for uniruled/adjunction contact manifolds, which is similar to the Seifert genus of knots.
  All strongly fillable planar contact 3-manifolds are uniruled and adjunction of complexity zero. We illustrate in Remark \ref{rmk: unifying picture} that many previously known Betti finite type examples have complexity zero.

We can give explicit bounds of the topological invariants in Theorem  \ref{t:uniruled} in terms of the genus invariants of the caps.
These bounds give an alternative proof of the following obstruction of planarity and generalize it to a larger class of contact 3-manifolds,
namely, uniruled/adjunction contact manifolds with complexity zero (See Corollary \ref{lem: strong extension of Wand}).

\begin{corr}[Wand~\cite{Wa12}]\label{cor:Wand}
For a given planar contact 3-manifold, $e+\sigma$ of any two  strong fillings of the contact 3-manifold are equal to each other.
\end{corr}

 On the other hand, we  construct various uniruled/adjunction contact manifolds with positive complexity.
For many of them, we also construct the open books supporting the concave boundary.
Contrary to Corollary \ref{cor:Wand}, $e+\sigma$ of fillings of some of these examples is not  constant even for Stein fillings
(see Example~\ref{ex:e+sign}).
Similar to the fact that planar contact manifolds are uniruled and adjunction of complexity zero,
we suspect that spinal planar contact manifolds \cite{LVhmW}, where the base genus of the  spine is allowed to be positive,  have a close relation to uniruled/adjunction contact manifolds with positive complexity.
On the other hand, it would be  interesting to explore the relation between uniruled/adjunction contact manifolds and  the results in \cite{We13}, \cite{We13-2} and \cite{We14}.

The rest of the paper is organized as follows.
Section \ref{s:Preliminary} reviews some knowledge of uniruled/CY manifolds, the relative de Rham theory and $J$-holomorphic neck-stretching.
Theorem \ref{t:CY} and Theorem \ref{t:lagrangian} are proved in Section \ref{s:CY}.
We also construct many Calabi-Yau caps with explicit open book description as well as some plumbing Calabi-Yau caps.
The proof of Theorem \ref{t:uniruled} together with its adjunction analogue are contained in Section \ref{s:UniAdj}.
We introduce genus invariants and the complexity invariant for uniruled/adjunction caps as well as illustrating its relation to planar contact manifolds and give
explicit topological bounds for fillings based on these invariants in Subsection \ref{ss:genusBound}.
We end with some explicit constructions of  uniruled/adjunction caps.

\subsection*{Acknowledgements}
 The authors are grateful to R. Inanc Baykur for his interest in our work and sharing with us his results in \cite{BaMoVHM14} with Naoyuki Monden and Jeremy Van Horn-Morris.
They also thank Anar Akhmedov for helpful discussions, Chris Wendl for informing them about Conjecture \ref{cj:lagrangian}
and an anonymous referee for a careful reading of the first version.
The third author would like to thank Selman Akbulut for fruitful conversations.
He also thanks Michigan State university and University of Minnesota, where this work was partly done, for their hospitality.
This is a substantially revised version of the one posted on arXiv in December, 2014.
The  second author gave a talk about the new results in this version at the Harvard Gauge theory seminar and
appreciates the discussion with Daniel Cristofaro-Gardiner and Peter Kronheimer.
Finally, the authors are very grateful to the referee whose numerous suggestions have greatly improved the presentation of this paper.

\section{General discussion}\label{s:Preliminary}

\subsection{Uniruled manifolds and Calabi-Yau manifolds}\label{ss:uniruledCY}

Let $X$ be a closed, oriented smooth 4-manifold.
Let ${\mathcal E}_X$ be the set of cohomology
classes whose Poincar\'e duals are represented by smoothly embedded
spheres of self-intersection $-1$. $X$ is said to be (smoothly)
minimal if ${\mathcal E}_X$ is the empty set.
Equivalently, $X$ is minimal if it is not the connected sum of
another manifold  with $\overline{\mathbb {CP}^2}$.

\medskip
Suppose $\omega$ is a symplectic form compatible with the orientation.
$(X,\omega)$  is said to be (symplectically) minimal if ${\mathcal
E}_{\omega}$ is empty, where
$${\mathcal E}_{\omega}=\{E\in {\mathcal
E}_X|\hbox{ $E$ is represented by an embedded $\omega-$symplectic
sphere}\}.$$
We say that $(Z, \tau)$
is a minimal model of $(X, \omega)$ if $(Z, \tau)$ is  minimal and $(X, \omega)$ is
a symplectic blow up of $(Z, \tau)$.
A basic fact proved using Taubes' SW theory (\cite{Taubes}, \cite{LiLiu}, \cite{Li99})  is:
  ${\mathcal E}_{\omega}$ is
empty if and only if ${\mathcal E}_X$ is empty. In other words, $(X,
\omega)$ is symplectically minimal if and only if $X$ is smoothly
minimal.

For minimal $(X,\omega)$,  the Kodaira dimension of
$(X,\omega)$
 is defined in the following way in \cite{Li06s} (see also \cite{McSa96} \cite{Le97}):

$$
\kappa^s(X,\omega)=\begin{cases} \begin{array}{lll}
-\infty & \hbox{ if $K_{\omega}\cdot [\omega]<0$ or} & K_{\omega}\cdot K_{\omega}<0,\\
0& \hbox{ if $K_{\omega}\cdot [\omega]=0$ and} & K_{\omega}\cdot K_{\omega}=0,\\
1& \hbox{ if $K_{\omega}\cdot [\omega]> 0$ and} & K_{\omega}\cdot K_{\omega}=0,\\
2& \hbox{ if $K_{\omega}\cdot [\omega]>0$ and} & K_{\omega}\cdot K_{\omega}>0.\\
\end{array}
\end{cases}
$$
Here $K_{\omega}$ is defined as the first Chern class of the
cotangent bundle for any almost complex structure compatible with
$\omega$.

The invariant $\kappa^s$ is well defined  since there does not
exist a minimal $(X, \omega)$ with $$K_{\omega}\cdot [\omega]=0,
\quad\hbox{and}\quad  K_{\omega}\cdot K_{\omega}>0.$$
This again follows from Taubes' SW theory \cite{Li06s}.
Moreover,
 $\kappa^s$ is independent of $\omega$, so it is an oriented diffeomorphism invariant of $X$.

 The Kodaira dimension of a non-minimal manifold is defined to be
that of any of its minimal models.
This definition is well-defined and independent of choice of minimal model so $\kappa^s(X, \omega)$ is well-defined for any $(X, \omega)$ (cf. \cite{Li06s}).

\begin{defn}
 Let $(X,\omega)$ be a not necessarily minimal closed symplectic four manifold.
 We call $(X,\omega)$ a {\bf symplectic Calabi-Yau surface} (resp. {\bf symplectic uniruled manifold}) if $\kappa^s(X, \omega)=0$ (resp. $\kappa^s(X, \omega)= -\infty$).
\end{defn}

We sometimes simply call a symplectic Calabi-Yau surface a Calabi-Yau surface and a symplectic uniruled manifold a uniruled manifold.
A minimal symplectic Calabi-Yau surface has torsion first Chern class.
The first author proved in \cite{Li06} the following theorem for symplectic Calabi-Yau surfaces (cf. also \cite{Ba08}).

\begin{thm}[Theorem 1.1 of \cite{Li06}]\label{t:CYhomologyBound}
 If $(X,\omega)$ is a minimal symplectic Calabi-Yau surface, then its rational homology is the same as that of the $K3$ surface, the Enriques surface or a torus bundle over torus.
\end{thm}

\begin{rmk}\label{r:integralK3}
If the first integral homology of a rational homology $K3$ symplectic Calabi-Yau surface is not trivial, it admits a finite cover which is also a symplectic Calabi-Yau surface but with Euler characteristic larger than that of the $K3$ surface. Hence, Theorem \ref{t:CYhomologyBound} implies that a  symplectic Calabi-Yau surface  with the  rational homology of the  $K3$ surface is an integral homology $K3$ surface.
\end{rmk}

For uniruled manifolds, Liu and independently Ohta and Ono proved the following smooth classification.
For the symplectic classification, see also \cite{Mc90}.

\begin{thm}[\cite{Liu96} or \cite{OhOn96}]\label{t:Uniruled}
 If $(X,\omega)$ is a minimal uniruled manifold, $X$ is $\mathbb {CP}^2, S^2\times S^2$ or an $S^2-$bundle over a Riemann surface
of positive genus.
\end{thm}

For a minimal uniruled manifold $(X,\omega)$, the base genus is defined to be zero if $X$ is $\mathbb {CP}^2$ or $S^2\times S^2$, otherwise it is defined to be the genus of the base
as an $S^2$-bundle.
For a general uniruled manifold $(X,\omega)$, the base genus is defined to be the base genus of any of its minimal models.

In some sense, Calabi-Yau manifolds and uniruled manifolds capture most of the known rigidity results for closed symplectic four manifolds.
This point of view motivates the definitions of Calabi-Yau caps and uniruled caps to obtain rigidity results for fillings in this paper.

We end this subsection with the following lemmas.

\begin{lemma}\label{l:genusUpperBound}
 Let $(X,\omega)$ be a uniruled manifold and $u:\Sigma_g \to X$ be a continuous map from an oriented surface of genus $g$ to $X$ with $(u_*[\Sigma_g])^2>0$.
 Then the base genus of $X$ is less than or equal to $g$.
\end{lemma}

\begin{proof}
 Without loss of generality, we can assume that $X$ is not of genus $0$.
 We further  assume that $X$ is not minimal since blowing-down does not change the base genus.
 Since $X$ is not minimal,  it can be obtained from $r$-times blow-ups of a product manifold $S^2 \times \Sigma_h$ for some genus $h$ surface \cite[Section 1.2.1]{FrMo}.
 Let $f$ be the class of $[S^2 \times p]$ for some $p \in \Sigma_h$ and $s$ the class of $[p \times \Sigma_h]$ for some $p \in S^2$.
  Notice that,  as a smooth manifold $X$ admits a smooth projection $\pi: X \to \Sigma_h$ with spherical fibers in the class $f$.
 We will show that  the induced  projection $\pi \circ u: \Sigma_g \to \Sigma_h$ has non-zero degree and hence $g \ge h$.

 By abuse of notation, $H_2(X,\mathbb{Z})$ is generated by $f,s,e_1,\dots,e_r$, where the $e_i$ are the classes of the exceptional spheres.
 Since $(a_0 f+ a_1 e_1+\dots+a_re_r)^2 \le 0$ for all $a_i \in \mathbb{Z}$,
 $u_*[\Sigma_g]$ has non-zero coefficient in $s$ when written as a linear combination over the basis $\{f,s,e_1,\dots,e_r\}$.
 By observing that  this non-zero coefficient is precisely the degree of
 the map $\pi \circ u$, we conclude that $h \le g$.

\end{proof}

\begin{lemma}[Proposition $3.14$ of \cite{LiZh09}]\label{l:weiyi}
 Let $(X,\omega)$ be a non-minimal uniruled manifold and $D$ a symplectic submanifold with positive genus.
 If $[D] \cdot e > 0$ for all exceptional classes $e$ in $X$, then $(K_{\omega}+[D])^2 \ge 0$.
\end{lemma}


\subsection{Relative cohomology pairing}\label{ss:RelativePairing}

In this subsection, we recall the relative de Rham theory and illustate the well-definedness of several pairings (see eg. \cite{BottTu}).

Given a smooth manifold with boundary $X$, the relative cochains $\mathcal{C}_k$ consist of pairs $(\beta,\alpha)$, where $\beta$ is a $k$-form on $X$ and $\alpha$ is a $(k-1)$-form on $\partial X$.
The differential $d$ is defined as $d(\beta,\alpha)=(d\beta,\beta|_{\partial X}-d\alpha)$.
It is easy to see that $d \circ d=0$ and it forms a cohomology isomorphic to the usual relative cohomology $H^*(X,\partial X; \mathbb{R})$.

This formulation of relative cohomology can be translated to compactly supported cohomology of $X-\partial X$.
For $(\beta,\alpha)$ a cochain in $\mathcal{C}_k$, we consider a collar $(0,1] \times \partial X$ of $\partial X$ in $X$.
We choose a cutoff function $\chi : (0,1] \times \partial X \to \mathbb{R}$.
We want $\chi(r,x)=\chi(r)$ for $(r,x) \in (0,1] \times \partial X$ such that $\chi(r)=0$ near $r=0$ and $\chi(r)=1$ near $r=1$.
Extending by $0$, $\chi \alpha$ is a $k-1$-form on $X$ which we also denote as $\alpha^c$.
Then $(\beta,\alpha)-d(\chi \alpha,0)=(\beta-d(\chi \alpha),0)$ is another chain level representative of $[(\beta,\alpha)]$ which has compact support in $X-\partial X$.
One can show that this translation induce an isomorphism from relative cohomology $H^k(X,\partial X, \mathbb{R})$ to the compactly supported cohomology $H^k_{cpt}(X-\partial X, \mathbb{R})$.

We assume that $X$ has dimension $4$ and is connected, oriented.
Consider the following pairing:
$$H^2(X;\mathbb{R})\times H^2(X,\partial X; \mathbb{R})\to H^4(X,\partial X; \mathbb{R})=\mathbb R$$
$$([A],    [(B, b)])\to  \int_X A \wedge B -\int_{\partial X} A \wedge b,$$
where the integral on $\partial X$ is taken with the Stokes boundary orientation.
To see it is independent of $A$, we check that by Stokes Theorem applied to $X$,
$$\int_X du \wedge B -\int_{\partial X} du  \wedge b= \int_X d(du\wedge b)-\int_{\partial X} du \wedge b =0.$$
To see it is independent of $(B, b)$, we check that by Stokes Theorem applied to both $X$ and $\partial X$,
$$\int_X A\wedge d\beta-\int_{\partial X} A \wedge (\beta-d\alpha)=\int_{\partial X} A\wedge d\alpha= \int_{\partial X} d(A\wedge \alpha)=0.$$
Notice that, this pairing is translated from the usual pairing bewteen
cohomololgy and compactly supported cohomology (see eg. \cite{BottTu}).
By Lefschetz duality, this pairing is non-degenerate.

Consider the following portion of the long exact sequence
\begin{equation} \label{rel seq} \cdots H^1(\partial X;\mathbb R)\stackrel{\partial} {\to} H^2(X,\partial X; \mathbb{R})\stackrel{f}{\to} H^2(X;\mathbb R)\stackrel{r}{\to}
H^2(\partial X;\mathbb R)\cdots,
\end{equation}
where at the form level,  $f$ sends $(A, a)$ to $A$,  and $r$ sends $B$ to $B|_{\partial X}$. Via the forgetful map $f$, we also have
the pairing
$$H^2(X,\partial X; \mathbb{R})\times H^2(X,\partial X; \mathbb{R})\to H^4(X,\partial X; \mathbb{R})=\mathbb R.$$
However, this pairing in general has a kernel.

\begin{lemma}\label{rel rel}
The pairing $[(A, a)] \cdot [(B, b)]$ is independent  of $a$ and $b$.
\end{lemma}
\begin{proof} If $b'$ is another primitive of $B$, then $b-b'$ is closed. Hence
$$\int_{\partial X}  A \wedge (b-b')=\int_{\partial X} da \wedge (b-b')=0.$$
The same argument applies to the choice of $a$.
\end{proof}

The kernel is contained in the image of the boundary homomorphism $\partial: H^1(\partial X;\mathbb R)\to H^2(X,\partial X; \mathbb{R})$ by the non-degeneracy of the pairing $H^2(X;\mathbb{R})$ and $H^2(X,\partial X; \mathbb{R})$.
Actually, the kernel is exactly the image because it is easy to see that it pairs everything to be zero.

From now on, we use $P$ or $N$ instead of $X$ to denote a manifold with boundary, depending on whether it is a cap or a filling.
The following simple lemma is the key to relate the caps with the closed manifold.

\begin{lemma}\label{l:splittingCohomologyPairing}
 Let $(P,\omega_P)$ be a symplectic cap of $(Y,\xi)$ with a Liouville one form $\alpha_P$,
 and $(N,\omega_N)$ be a symplectic filling of $(Y,\xi)$ with a Liouville  one form $\alpha_N$.
 Let $X=N \cup_Y P$ which is a closed manifold.
 Then for sufficiently large $t>0$, there is a symplectic form $\omega$ on $X$
 such that $c_1(X) \cdot \omega = c_1(N)\cdot[(\omega_N,\alpha_N)]+c_1(P)\cdot t[(\omega_P,\alpha_P)]$
 and $\omega|_N=\omega_N$.
\end{lemma}

\begin{proof}
 We identify $(\partial N, \ker (\alpha_N))$ and $(\partial P, \ker (\alpha_P))$ by a contactomorphism $\Phi$.
 There is a global positive function $f_{\alpha_P}:\partial N \to \mathbb{R}$ such that $\Phi^* \alpha_P =f_N \alpha_N$.
 When $t >0$ is large, $\Phi^* t\alpha_P =f_{t \alpha_P} \alpha_N$ is such that $f_{t \alpha_P}(x) >1$ for all $x \in \partial N$.
 We fix such a choice of $t$.
 Consider the symplectization $(\mathbb{R} \times \partial N, d(e^r\alpha_N), e^r\alpha_N)$ where $r \in \mathbb{R}$.
 Let $SY=\{(r,x) \in \mathbb{R} \times \partial N| 1 \le e^r \le f_{t \alpha_P}(x) \}$
 We equip it with the restricted symplectic form and one form and call it $(SY,\omega_{SY},\alpha_{SY})$.
 We can glue $(N,\omega_N,\alpha_N)$ and $(P,t\omega_P,t\alpha_P)$ by inserting
 $(SY,\omega_{SY},\alpha_{SY})$, see \cite{Et98}.
 Notice that $\alpha_{SY}$ is a globally defined primitive of $\omega_{SY}$ on $SY$ and coincides with $\alpha_N$ and $t\alpha_P$ on its two boundary components, respectively.
Let the resulting manifold be $(X,\omega)$, which is the union of $N$, $SY$ and $P$.
By multiplying a cutoff function and by abuse of notation, we can extend $\alpha_{SY}$ to be a one form supported in a neighborhood of $SY \subset X$ such
that $\alpha_{SY}|_N=\alpha_N^c$ and $\alpha_{SY}|_P=t\alpha_P^c$, where $\alpha_N^c$ and $t\alpha_P^c$ are defined as in the third paragraph of this subsection.
Therefore, we have
\begin{eqnarray*}
c_1(X)\cdot[\omega] & =& c_1(X)\cdot[\omega-d\alpha_{SY}]\\
& =&c_1(X)\cdot[\omega_N-d\alpha_N^c]+c_1(X)\cdot t[\omega_P-d\alpha_P^c] \\
& =&c_1(N)\cdot[(\omega_N,\alpha_N)]+c_1(P)\cdot t[(\omega_P,\alpha_P)]
\end{eqnarray*}
which is simply the sum of the pairings from the cap and the filling.
This is clearly true for any $t$ sufficiently large.
It completes the proof.
\end{proof}

The following properties are also useful.

\begin{lemma}\label{l:relCohProperties}
 Let $(P,\omega)$ be a symplectic cap and $\alpha$ a choice of Liouville one form.
 Let $\Sigma \subset P$ be a compact embedded surface with boundary $\partial \Sigma \subset \partial P$.
 Then the followings are true.
 \begin{itemize}
  \item $[(\omega,\alpha)]^2 > 0$,
  \item $[\Sigma] \cdot PD([(\omega,\alpha)])=\int_{\Sigma} \omega -\int_{\partial \Sigma} \alpha$, where $\partial \Sigma$ is equipped with the Stokes orientation.
  \item if $c_1(P) \cdot [(\omega,\alpha)] =0$ and $c_1(P) \neq 0 \in H^2(P,\mathbb{R})$, then there is a small perturbation $(\omega',\alpha')$ of $(\omega,\alpha)$
  such that $(P,\omega')$ is a cap, $\alpha'$ is a Liouville one form of $\omega'$ and $c_1(P) \cdot [(\omega',\alpha')] >0$.
 \end{itemize}
\end{lemma}

\begin{proof}
 For the first bullet, we have $[(\omega,\alpha)]^2=\int_{P} \omega^2 -\int_{\partial P} \alpha \wedge \omega$.
 The first term is positive because the orientation of $P$ is always chosen to be compatible with $\omega^2$.
 On the other hand, $\int_{\partial P} \alpha \wedge \omega < 0$ because the orientation of $\partial P$ as a contact manifold is determined by $\iota_V (\omega^2)=2 \alpha \wedge \omega$
 for an inward pointing vector field $V$,
 while the Stokes orientation of $\partial P$ is determined by $\iota_{V_{Stokes}} (\omega^2)$ for an outward pointing vector field $V_{Stokes}$.
 Therefore, $[(\omega,\alpha)]^2 >0$.

 The second bullet follows from  definition.

 For the last bullet, if $c_1(P) \neq 0$ we can find a relative cohomology class $[(A,a)]$ pairs positively with $c_1(P)$ by
 the non-degeneracy of the pairing between absolute cohomology and relative cohomology.
 Let $(A,a)$ be a chain level representative of $[(A,a)]$.
 The result follows by adding $c(A,a)$ to $(\omega,\alpha)$ for some small $c>0$.
\end{proof}

\begin{rmk}
We could have defined Calabi-Yau cap by the equation $c_1(P) \cdot [(\omega,\alpha)] =0$ instead of $c_1(P) = 0 \in H^2(P,\mathbb{R})$.
In this case, the third bullet of Lemma \ref{l:relCohProperties} implies that if $c_1(P) $ is not torsion, then we can deform the cap to a uniruled cap
which gives  stronger restrictions to fillings.
Therefore, we stick to our defintion of Calabi-Yau caps.
\end{rmk}

\subsection{Neck-stretching basic}

Let $(X,\omega)$ be a closed symplectic four manifold and $\mathcal{J}$ the space of $\omega$-compatible almost complex structures.
We recall some basic Gromov-Witten theory and neck-stretching techniques (See \cite{ElGiHo00}, \cite{BEHWZ03} and \cite{McL14} for more comprehensive account).
For any $J \in \mathcal{J}$ and a (connected) tree $T$ with $|T|$  (finite) vertices, we call $u=(u_i)_{i=1}^{|T|}$ a {\bf closed genus $0$ nodal} $J$-holomorphic map modeled on $T$ if
for each vertex $v_i$ of $T$, there exists a $J$-holomorphic map $u_i: \mathbb{C}P^1 \to (X,\omega,J)$
such that the intersection pattern of $\{u_i(\mathbb{C}P^1)\}$ is given by the tree $T$
(i.e. an edge joining two vertices corresponds to an intersection between the corresponding $u_i(\mathbb{C}P^1)$).
A closed genus $0$ nodal $J$-holomorphic map $u$ is one that modeled on some tree $T$.
In this case, we also call $u$ a closed genus $0$ nodal $J$-holomorphic representative for the homology class $[u]=\sum u_{i*}[\mathbb{C}P^1]$.
The following proposition is well-known and readers are referred to \cite{Mc90}, \cite{McL14} and the references therein.

\begin{prop}\label{p:GWexceptional}
 Let $e$ be an exceptional class of $(X,\omega)$.
 Then for any $J \in \mathcal{J}$, there is a closed genus $0$ nodal $J$-holomorphic representative $u$ of $e$.
\end{prop}

Let $(Y,\xi)$ be a separating contact hypersurface in $(X,\omega)$.
Here, we mean that there is a Liouville vector $V$ defined near $Y$ such that $V$ is transversal to $Y$.
We denote the Liouville one form $\alpha$ and $\xi=\ker(\alpha)$.
Let $(P,\omega|_P)$ be the cap and $(N,\omega|_N)$ be the filling of $Y$ obtained by cutting $X$ along $Y$.
Let $Int(P)$ and $Int(N)$ be the corresponding interiors.
We call an $\omega|_P$-compatible almost complex structure $J^{\infty}$ is {\bf cylindrical} if there is a collar
$([0,\epsilon) \times Y, \omega=d(e^r \alpha),\alpha)$ of $Y$ in $P$ such that $J^{\infty}(\xi)=\xi$, $J^{\infty}(\partial_r)=\partial_{Reeb}$, $J^{\infty}(\partial_{Reeb})=-\partial_r$
and $J^{\infty}|_{\xi}$ is translational invariant with respect to $r$,
where $r \in [0,\epsilon)$ and $\partial_{Reeb}$ is the Reeb vector field of $\alpha$.
Any cylindrical $J^{\infty}$ can be extended to a cylindrical compatible almost complex structure on the symplectic completion of $P$ and
we denote this extension as $J^{\infty}$ by abuse of notation.
For a cylindrical $J^{\infty}$, we call $u^{\infty}=(u^{\infty}_i)$ a
{\bf genus $0$ top building} if the $u^{\infty}_i$ are $J^{\infty}$-holomorphic maps (finitely many and possibly empty) from genus $0$ puntured Riemann surfaces
$\Sigma_i$ to the symplectic completion of $P$ with puntures aysmptotic to Reeb orbits of $Y$ on its negative end.
Similarly, we can think of $u^{\infty}_{i*}([\Sigma_i]) \in H_2(P,\partial P)$ and we denote $\sum u^{\infty}_{i*}([\Sigma_i])$ as $[u^{\infty}]$.

\begin{prop}\label{p:neck}(See \cite{ElGiHo00}, \cite{BEHWZ03} and cf. \cite{McL14})
Let $[D]$ be a homology class in $H_2(P,\mathbb{Z})$.
For a tubular neighborhood $\mathcal{N}$ of $Y$ in $X$,
there is a sequence $J^k \in \mathcal{J}$ and an $\omega|_P$-compatible cylindrical almost complex structure $J^{\infty}$
on $P$ such that the following holds.
\begin{itemize}
 \item $J|_{P-\mathcal{N}}=J^{\infty}|_{P-\mathcal{N}}=J^k|_{P-\mathcal{N}}$ for all $k$,
 \item if, for all $k$, $u^k=(u^k_i)$ is a closed genus $0$ nodal $J^k$-holomorphic map representing the same homology class in $X$,
 then there is a genus $0$ top building $u^{\infty}=(u^{\infty}_i)$ to the symplectic completion of $P$, and
 \item $[D] \cdot [u^k]= [D] \cdot [u^{\infty}]$, where the first pairing takes place in $H_2(X)$ (ie. $[D]$ represents its image to $H_2(X)$)
 and the second is the pairing between $H_2(P)$ and $H_2(P, \partial P)$.
\end{itemize}

\end{prop}

\begin{proof}
 The first two bullets are part of the compactness result in symplectic field theory.
 The last bullet is true because the other buildings at the SFT limit of $u^k$ do not contribute to the intersection pairing with $[D]$.
\end{proof}

\begin{corr}\label{c:neck}
 Let $(P,\omega_P)$ be a symplectic cap of $(Y,\xi)$.
 Assume that $[(\omega,\alpha)]$ is a rational class and let $[D]\in H_2(P,\mathbb{Z})$ be the Lefschetz dual of $c[(\omega,\alpha)]$ for some $c >0$.
 Then for any minimal strong symplectic filling $N$ of $(Y,\xi)$ and any exceptional class $e$ in $X$, we have $[D] \cdot e >0$.
\end{corr}

\begin{proof}
 By Proposition \ref{p:GWexceptional}, there is a closed genus $0$ nodal representative $u$ for any exceptional class $e$ and any $J \in \mathcal{J}$.
 We apply Proposition \ref{p:GWexceptional} to the choice of $J^k$ in Proposition \ref{p:neck}, then we get a $J^{\infty}$ genus $0$ top building $u^{\infty}=(u^{\infty}_i)$.
 Notice that, the $u^{\infty}_i$ are $J^{\infty}$-holomorphic for all $i$ and they have punctures asymptotic to Reeb orbits.
 Near each puncture of $\Sigma_i$, the domain of $u_i^{\infty}$, we have a cylindrical-like conformal coordinate from which we can find a small circle $C$ sufficiently close to the puncture
 such that  $(\pi_Y \circ u_i^{\infty})^*\alpha|_C < 0$ with respect
 to the Stokes orientation of $C$ induced by $\Sigma_i$.
 Here, $\pi_Y$ is the projection from the negative end of the completion of $P$ to the factor $Y$.
 This is because we can find $C$ such that $\pi_Y \circ u_i^{\infty}(C)$ is close to a Reeb orbit of $Y$ and the Stokes orientation of $C$ is different from the orientation induced by $(\pi_Y \circ u_i^{\infty})^*\alpha$.
 Therefore, as in the calculation of first bullet of Lemma \ref{l:relCohProperties} and
 by the second bullet of Lemma \ref{l:relCohProperties}, we must have $u^{\infty}_{i*}[\Sigma_i] \cdot [D] > 0$.
 Therefore, Proposition \ref{p:neck} implies $[D] \cdot e >0$.

\end{proof}

\subsection{A general property for Betti finiteness}

For any contact 3-manifold of Stein,  exact, or strong Betti finite type, we have the following restriction for simply connected fillings.

\begin{prop}\label{l:finiteHomeo}
 If $(Y, \xi)$ is of Stein, exact, or strong Betti finite type, then $(Y, \xi)$
 has at most finitely many simply connected Stein, exact or minimal strong, respectively,  fillings up to homeomorphism.
\end{prop}

To prove this proposition, we introduce necessary definitions and a lemma.
For an integral symmetric bilinear form $Q: \mathbb{Z}^n\times \mathbb{Z}^n\to \mathbb{Z}$, let  $M_Q$ be a matrix presentation of $Q$,
and let $G_Q$ be the group presented by the matrix $M_Q$ (i.e. the cokernel of the homomorphism $\mathbb{Z}^n\to \mathbb{Z}^n$ given by $M_Q$).
Note that $G_Q$ is independent of the choice of $M_Q$. Let $r_{Q}$, $d_Q$ be the rank of  $G_Q$ and the number of elements of $\textnormal{Tor}\, (G_Q)$, respectively.

 Though the lemma below might be known to experts, we give a proof since we could not find any reference.
\begin{lemma}\label{lem: bilinear form}
For any finitely generated abelian group $G$ and any positive integer $n$, there exist at most finitely many isomorphism types of integral symmetric bilinear forms such that their matrix presentations present $G$ and have the size $n$.
\end{lemma}
\begin{proof}Let $d$ denote the number of elements of $\textnormal{Tor}\,(G)$, and put $r=\textnormal{rank}\, G$. We prove the claim by induction on the number $r\geq 0$. The $r=0$ case follows from the finiteness of isomorphism types of intersection forms with non-zero determinant. For this fact, see Theorem 1.1 in Chapter 9 of \cite{Cas78}. Note that $d_Q=\textnormal{det}(M_Q)$ in this case.

Assuming the $r=k\geq 0$ case, we prove the $r=k+1$ case.  The condition $r\geq 1$ implies $\textnormal{det}(M_Q)=0$ for any intersection form $Q$ with $G_Q\cong G$. Therefore, there exist integral square matrixes $A, B$ with size $n$ and $|\textnormal{det}(A)|=|\textnormal{det}(B)|=1$ such that $AM_QB$ is a diagonal matrix which has a zero in a diagonal component. Using this fact, we easily see that there exists a primitive element $x\in \mathbb{Z}^n$ satisfying $Q(x, y)=0$ for any $y\in \mathbb{Z}^n$. As a consequence, $Q$ has the orthogonal sum decomposition $Q=Q|_{\langle x \rangle}\oplus Q|_H$ for some subgroup $H$ of $\mathbb{Z}^n$. Since $G_Q\cong G_{Q|_{\langle x \rangle}}\oplus G_{Q|_{H}}$, we see $r_{Q|_H}=k$. Therefore, the assumption on the induction shows the $r=k+1$ case.
\end{proof}

\begin{proof}[Proof of Proposition \ref{l:finiteHomeo}]Let $(Y,\xi)$ be a contact 3-manifold of Stein Betti finite type (resp.\ exact or strong Betti finite type). The intersection form $Q$ of any simply connected compact 4-manifold with the boundary $Y$ satisfies $G_{Q}\cong H_1(Y;\mathbb{Z})$ (cf.\ \cite{GS99}). By the assumption, there are only finitely many possible values of $b_2$ for Stein (resp.\ exact or minimal strong) fillings of $(Y,\xi)$. Therefore, according to  Lemma~\ref{lem: bilinear form}, there are only finitely many possible intersection forms of such Stein (resp.\ exact or minimal strong) fillings. According to a theorem of Boyer (Corollary 0.4 in \cite{Bo86}), for a given connected oriented closed 3-manifold and  intersection form, there are at most finitely many topological types of simply connected 4-manifolds which realize the given boundary and the intersection form. Therefore the desired claim follows.
\end{proof}

\section{Calabi-Yau caps and exact fillings} \label{s:CY}

\subsection{Theorem \ref{t:CY} and Theorem \ref{t:lagrangian}}

We prove Theorem \ref{t:CY}, Theorem \ref{t:lagrangian} and Corollary \ref{c:infiniteFamily} in this subsection.

\begin{proof} [Proof of Theorem \ref{t:CY}]
Let $(P,\omega_P)$ be a Calabi-Yau cap with a Liouville contact form $\alpha_P$ for the contact manifold $(Y, \xi)$. We must now establish that $(Y, \xi)$ is of exact Betti finite type.

 For an exact symplectic filling $(N,\omega_N)$ of $(Y,\xi)$, we also have a Liouville contact form $\alpha_N$ on $\partial N$ making
$c_1(N) \cdot [(\omega_N,\alpha_N)] =0$.
Since  $c_1(P) \cdot [(\omega_P,\alpha_P)] =0$,
   by Lemma \ref{l:splittingCohomologyPairing}, the glued closed symplectic manifold $(X,\omega)$ has $c_1(X) \cdot [\omega]=0$.
If $X$ is not minimal, we can blow down the exceptional spheres to obtain a minimal model.
Since blow-down increases $c_1(X) \cdot [\omega]$, we must have $X$ being non-minimal uniruled or minimal symplectic Calabi-Yau.

In either case we will establish uniform bounds of Betti numbers of $X$ which depend only on $P$.
It then follows from the Mayer-Vietoris sequence there are uniform bounds of Betti numbers of $N$
depending only on $P$.

Let us  start with the case that  $X$ is non-minimal uniruled.
We will first bound  $b_1(X)$ and $b_3(X)$.
Since $[(\omega_P,\alpha_P)]^2>0$ by Lemma \ref{l:relCohProperties} and $b_2^+(X)=1$, we have $b_2^+(P) = 1$ and $b_2^+(N)=0$.
On the other hand, we can find a closed oriented smoothly  embedded surface $S$ in $P$ whose homology class is the Lefschetz dual of $c\,[(\omega_P,\alpha_P)]$ for some constant $c$,
by possibly perturbing $[(\omega_P,\alpha_P)]$ to a rational class (cf. \cite{GS99}, Remark 1.2.4).
Then $[S]^2 > 0$ and one should think that $S$ is chosen before gluing with $N$.
Let $g(S)$ denote the genus of $S$.
By Lemma \ref{l:genusUpperBound}, the base genus of $X$ is less than or equal to $g(S)$ and this gives an upper bound for $b_1(X)$ and hence $b_3(X)$ depending only on $P$.

To bound $b_2(X)=b_2^+(X)+b_2^-(X)=1+b_2^-(X)$, it suffices to bound
$b_2^-(X)$. And to bound $b_2^-(X)$ it suffices to bound $c_1(X)^2$ from below. This simply follows from
the just established bound on $b_1(X)$ and the relation
$$c_1(X)^2=2e(X)+3\sigma(X)=4-4b_1(X) +5b_2^+(X)-b_2^-(X)=9-4b_1(X)-b_2^-(X),$$
where we again used the fact $b_2^+(X)=1$ in the last equality.

To bound $c_1(X)^2$ we need to choose the surface $S$ more carefully.
We observe that $S$ can be chosen such that $ -c_1(P)\cdot[S]+[S]^2  \geq 0$ by  choosing a larger
$c$.
We further assume that
$[S]^2 \ge g(S)-1$, by possibly choosing an even  larger $c$.
The reason is, once $S$ is chosen as above, we consider $\nu$ distinct copies of embedded surfaces representing $[S]$ by local perturbation of $S$.
We assume each pair of these $\nu$ copies are  intersecting transversally and positively.
After resolving all the positive  intersection points for these $\nu$ copies, we get an embedded surface of genus $\nu g(S)+\frac{(\nu-1)\nu}{2}[S]^2-(\nu-1)$ with self-intersection $\nu^2[S]^2$.
When $\nu$ is large, we get an embedded surface with homology class being a positive multiple of $cPD[(\omega_P,\alpha_P)]$ such that
the self-intersection number is greater than the genus.
In summary,  the surface $S$ is assumed to satisfy  $[S]^2 \ge g(S)-1$ and $ -c_1(P)\cdot[S]+[S]^2  \geq 0$.

Since $\omega_N$ is exact,  $\omega_P-d\alpha_P^c$ in Lemma \ref{l:splittingCohomologyPairing} viewed as a closed two form on $X$ represents the same cohomology class as $\omega$.
Therefore, $S$ is the Poincar\'e dual of $c[\omega]$ and any exceptional class pairs positively with $[S]$ in $X$.
By \cite{LiLi02}, there is a symplectic surface $\widetilde{S} \subset X$ representing $[S]$ and the genus
$\widetilde{g}$ of $\widetilde{S}$ is less than or equal to $g(S)$ according the genus minimizing property of symplectic surfaces (\cite{KrMr94}, \cite{MoSzTa96}). Notice that $\tilde g$ is determined by
the adjunction formula:
$$2\tilde g -2 = -c_1(X)\cdot [S]+[S]^2=-c_1(P)\cdot[S]+[S]^2.$$
Thus we have $\widetilde{g} > 0$.
We may assume that $X$ is not minimal, otherwise there is nothing to prove.
By Lemma \ref{l:weiyi}, we have $(-c_1(X)+[\widetilde{S}])^2 \ge 0$.
Notice that
$$(-c_1(X)+[\widetilde{S}])^2=c_1(X)^2 + 2(-c_1(X) \cdot [S]+[S]^2)-[S]^2=c_1(X)^2+2(2\widetilde{g}-2)-s,$$ where the second equality is by adjunction and $s=[S]^2$.
Therefore, $$c_1(X)^2 \ge s-2(2\widetilde{g}-2) \ge s-2(2g(S)-2).$$
Since $g(S)$ and $s$  only depend on $S \subset P$, the lower bound of $c_1(X)^2$ is independent of $N$.
As a result, $b_2(X)$ is bounded.


Next, we assume $X$ is a minimal symplectic Calabi-Yau surface, Theorem \ref{t:CYhomologyBound} give a uniform bound on the Betti numbers of $X$ (which are actually  independent of $P$).
Hence, $b_1(N), b_2(N),b_3(N)$ are uniformly  bounded by the algebraic topology of $P$.
It finishes the proof of the first statement.
The second statement is straightforward  because $X$ must be a minimal symplectic Calabi-Yau surface.



\end{proof}

We now study the homology and cohomology of exact fillings of the unit cotangent bundles of surfaces.
We start with a lemma.

\begin{lemma}\label{l:constructLag}
There is a symplectic $K3$ surface which has $g$ disjoint copies of embedded Lagrangian tori representing the same homology class and each of which intersects  an embedded Lagrangian sphere transversally at one  point.
\end{lemma}

\begin{proof}
We first recall a construction of a $K3$ surface \cite{GS99}.
Let $(T^4,\omega_T)$ be the four torus equipped with the quotient Kahler form induced by the standard Kahler form on $\mathbb{C}^2$ quotient by $\mathbb{Z}^4$.
There is an involution $I: T^4 \to T^4$ defined by $I(z,w)=(-z,-w)$ which has $16$ fixed points.
The quotient $X'=T^4/((z,w) \sim I(z,w))$ has $16$ singular points.
After resolving these $16$ singular points, which results in $16$ spheres of self-intersection $-2$, we get our $K3$ surface $(X,\omega_X)$.

We can write $\omega_T$ as $dx \wedge dy +du \wedge dv$.
Then there is a family of Lagrangian tori corresponding to $x,u$ coordinates (ie. tangent space spanned by $\partial_x,\partial_u$)
and another family corresponding to $y, v$ coordinates.
We call them $xu$-tori and $yv$-tori.

For any $g >1$, pick $g$ disjoint $xu$-tori.
One can choose these $g$ Lagrangian $xu$-tori in a way that they avoids the $16$ fixed points and such that they
descend to disjoint embedded Lagrangian tori in the quotient $X'$.
These $16$ tori in $X'$ can then be lifted to $X$.

We consider a $yv$-torus that passes through $4$ of the $16$ fixed points.
When descended to $X'$, the image of this Lagrangian torus becomes a orbifold Lagrangian sphere with four orbifold points.
We call the orbifold sphere $yv$-orbifold sphere.
We claim that we can resolve the $16$ orbifold points of $X'$ in a way that $yv$-orbifold sphere lifts to an embedded Lagrangian sphere in $X$.
Once we established the claim, it is clear that this Lagrangian sphere together with the $g$ Lagrangian tori are the Lagrangians in $X$ we want.

To prove the claim, it suffices to understand the local model for the symplectic resolution at the orbifold points.
It turns out that the resolution is symplectically the same as replacing a neighborhood of an orbifold point and gluing back a neighborhood of zero section of
$T^*S^2$, where the $yv$-orbifold sphere near the orbifold point is identified with a fiber of $T^*S^2$ after the gluing.
Hence, the $yv$-orbifold sphere can be lifted to $X$ by extending the fibers across the zero section of $T^*S^2$.

After explaining the effect of the resolution and why the claim follows from it, we now explain how the surgery goes, which turns out to be a routine calculation.
We start with a model for $T^*S^2$.
Let $U_1=U_2=\mathbb{C}$ and $\phi_{12}:U_1\backslash \{0\} \to U_2\backslash \{0\}$ be $\phi_{12}(z)=\frac{1}{z}$.
Then $\Phi_{12}=(\phi_{12}^*)^{-1}:T^*(U_1\backslash \{0\}) \to T^*(U_2\backslash \{0\})$ is given by $\Phi_{12}(z,w)=(\frac{1}{z},-\overline{z}^2w)$,
where $z \in U_1$ and $w \in \mathbb{C}$.
The standard Liouville one form on $T^*U_1$ is given by $\lambda_1=p_1dq_1+p_2dq_2=\frac{1}{2}(wd\bar{z}+\bar{w}dz)$, where the identification is $z=q_1+iq_2,w=p_1+ip_2$.
The Liouville form on $T^*U_2$ is similar.
These give the description of $T^*S^2$.

We define a double covering $\rho_j: U_j \times (\mathbb{C} \backslash \{0\}) \to (T^*U_j\backslash U_j)$ by
$$\rho_j(\widetilde{z},\widetilde{w})=(\widetilde{z},i \overline{\widetilde{w}^2})$$
for $j=1,2$.
This double covering can be globalized using the transition map $\widetilde{\Phi}_{12}: (U_1 \backslash \{0\}) \times (\mathbb{C} \backslash \{0\}) \to (U_2 \backslash \{0\}) \times (\mathbb{C} \backslash \{0\})$ given by
$$\widetilde{\Phi}_{12}(\widetilde{z},\widetilde{w})=(\frac{1}{\widetilde{z}},i\widetilde{z}\widetilde{w})$$
in the sense that $\rho_2 \circ \widetilde{\Phi}_{12}=\Phi_{12} \circ \rho_1$ over $(U_1 \backslash \{0\}) \times (\mathbb{C} \backslash \{0\})$.
Clearly, $U_j \times (\mathbb{C} \backslash \{0\})$ together with $\widetilde{\Phi}_{12}$ are charts and transition function of the $O(1)$ bundle over $\mathbb{CP}^1$ away from the zero section.
Moreover, $\rho_j$ determines a double covering to $T^*S^2 \backslash S^2$.
The pull-back one form is given by
$$\widetilde{\lambda}_1=\rho_1^*\lambda_1=\frac{1}{2}(i (\overline{\widetilde{w}})^2d\overline{\widetilde{z}}-i\widetilde{w}^2d\widetilde{z})$$

We define diffeomorphisms $\Psi_1: (\mathbb{C}\backslash \{0\}) \times \mathbb{C} \to U_1 \times (\mathbb{C} \backslash \{0\})$ by
$$\Psi_1(\hat{z},\hat{w})=(\frac{\hat{w}}{\hat{z}},\hat{z})$$
and $\Psi_2: \mathbb{C} \times (\mathbb{C} \backslash \{0\}) \to U_2 \times (\mathbb{C} \backslash \{0\})$ by
$$\Psi_2(\hat{z},\hat{w})=(\frac{\hat{z}}{\hat{w}},i\hat{w})$$
which satisfy $\widetilde{\Phi}_{12} \circ \Psi_1=\Psi_2$ over $(\mathbb{C}\backslash \{0\}) \times (\mathbb{C}\backslash \{0\})$.
Hence $\Psi_1,\Psi_2$ together give a diffeomorphism from $\mathbb{C}^2 \backslash \{0\}$ to $O(1)$-bundle of $\mathbb{CP}^1$ away from zero section.
The differential of the pull-back one form is given by
$$d(\Psi_1^*\widetilde{\lambda}_1)=-i(d\hat{z} \wedge d\hat{w}-d\overline{\hat{z}} \wedge d\overline{\hat{w}})$$
By letting $\hat{z}=u+ix,\hat{w}=y+iv$, we get $d(\Psi_1^*\widetilde{\lambda}_1)=2(dx \wedge dy +du \wedge dv)$ which is the standard symplectic form up to a constant multiple.
In particular, it coincide with the $\omega_T$ on $T^4$ near a fixed point.
The $yv$-torus near a fixed point corresponds to $\hat{z}=0$, which can be identified as a fiber of $U_2 \times (\mathbb{C} \backslash \{0\})$
under $\Psi_2$.
Finally, note that the involution on $T^4$ satisfies $\rho_j \circ \Psi_j \circ I=\rho_j \circ \Psi_j$ for both $j=1,2$, which tells us that
the $yv$-orbifold sphere near a orbifold point correspond to a fiber of $T^*U_2\backslash U_2$ as claimed.

\end{proof}

\begin{proof}[Proof of Theorem \ref{t:lagrangian}]
By Lemma \ref{l:constructLag}, we have $g$ Lagrangian tori representing the same class $A$ such that
each  transversally intersects a Lagrangian sphere in a symplectic $K3$ surface $X$.
Let the homology class of the sphere be $B$.
We smooth out the intersection points by local Lagrangian surgery \cite{Pol} and result in an embedded Lagrangian genus $g$ surface $L$.

Let $U$ be the unit cotangent disk bundle and identify it with a Weinstein neighborhood of $L$.
Then, the complement of the interior of $U$ gives a Calabi-Yau cap $P$ for $Y$.
Notice that $[L]$ is in the span of  $A$ and $B$ and the intersection form restricted to the subspace spanned by $A$ and $B$ is given by
\[ \left( \begin{array}{cc}
0 & 1 \\
1 & -2
\end{array} \right) \]
which is equivalent to
\[ H=\left( \begin{array}{cc}
0 & 1 \\
1 & 0
\end{array} \right). \]

As a result, the orthogonal complement of this subspace has intersection matrix $-2E_8 \oplus 2H$.
On the other hand, $(g-2)A-B$ is orthogonal to $[L]=gA+B$.
In other word, the bilinear form $-2E_8\oplus2H\oplus(2-2g)$ embeds into the intersection form of $P$, where $2-2g$ corresponds to the direction spanned by the class $(g-2)A-B$, which has self-intersection $2-2g$.

Note that $H_2(Y;\mathbb{Z})=H^1(Y;\mathbb{Z})=\mathbb{Z}^{2g}$.
As a circle bundle, the generators of $H_2(Y;\mathbb{Z})$ are given by a loop from the base $L$ times the circle fibers.
It is the boundary of the same loop of the base times the disk fiber in $U$, which means that $H_2(Y;\mathbb{Z}) \to H_2(U;\mathbb{Z})$ is a zero map.
From the long exact sequence
$$0 \to H_2(Y;\mathbb{Z}) \to H_2(U;\mathbb{Z}) \oplus H_2(P;\mathbb{Z}) \to H_2(X;\mathbb{Z})=\mathbb{Z}^{22}$$
we see that $H_2(Y;\mathbb{Z}) \to H_2(P;\mathbb{Z})$ is an injection.
Since $H_2(X;\mathbb{Z})$ has no torsion and $H_2(Y;\mathbb{Z})$ is free, both $H_2(U;\mathbb{Z})$ and $H_2(P;\mathbb{Z})$ do not have torsion.
Hence we know that $H_2(P;\mathbb{Z})=\mathbb{Z}^{2g+21}$
and the intersection matrix of $P$ is given by
$-2E_8\oplus2H\oplus(2-2g) \oplus (0)^{2g}$, where $(0)^{2g}$ corresponds to the subspace spanned by the image of $H_2(Y;\mathbb{Z})$.

From the intersection form of $P$, we see that $P$ cannot embed into any uniruled manifold or minimal symplectic Calabi-Yau surface
 other than a homology $K3$ surface.
Let $N$ be any exact filling of $Y$,
the glued symplectic manifold $P \cup N$ has to be a minimal integral homology $K3$ surface denoted as $\overline{X}$, by the first paragraph in the proof of Theorem \ref{t:CY} and Remark \ref{r:integralK3}.
It implies that $N$ has Euler characteristic $e(N)=2-2g$ and signature $\sigma(N)=1$.
In particular, $b_2(N) \ge 1$.

By the long exact sequence
$$H_4(\overline{X};\mathbb{Z})\to H_3(Y;\mathbb{Z}) \to H_3(N;\mathbb{Z}) \oplus H_3(P;\mathbb{Z}) \to 0$$
and the fact that the first map is always an isomorphism, we have $H_3(N;\mathbb{Z})=H_3(P;\mathbb{Z})=0$.
Next the long exact sequence
$$0 \to H_2(Y;\mathbb{Z}) \to H_2(N;\mathbb{Z}) \oplus H_2(P;\mathbb{Z}) \to \mathbb{Z}^{22}$$
tell us that $H_2(N;\mathbb{Z})=\mathbb{Z}$ or $H_2(N;\mathbb{Z})=0$ because $H_2(P;\mathbb{Z})=\mathbb{Z}^{2g+21}$ and $H_2(Y;\mathbb{Z})=\mathbb{Z}^{2g}$.
The latter one is ruled out by the fact that $b_2(N) \ge 1$ so $H_2(N;\mathbb{Z})=\mathbb{Z}$.

Note that the $-2E_8\oplus2H$ lattice from $P$ embed into the intersection form of $\overline{X}$.
Also, both the generator $[S_N]$ of  $H_2(N;\mathbb{Z})$ and the $(g-2)A-B$ class in $P$ lies in the orthogonal complement of
$-2E_8\oplus2H$ in $\overline{X}$.
By the classification of unimodular bilinear form, the orthogonal complement of $-2E_8\oplus2H$ in $\overline{X}$ is $H$.
This together with the fact that $[S_N]$ is orthogonal to $(g-2)A-B$ in $H$ implies that $[S_N]$ has self-intersection $k^2(2g-2)$ for some positive integer $k$
(because the primitive class orthogonal to $(g-2)A-B$ has self-intersection $2g-2$).

Finally, we want to determine $H_1(N;\mathbb{Z})$ using the fact that $[S_N]^2=k^2(2g-2)$.
From the long exact sequence
$$ 0 \to H^3(N;\mathbb{Z}) \oplus H^3(P;\mathbb{Z}) \to H^3(Y;\mathbb{Z}) \to H^4(\overline{X};\mathbb{Z})$$
and the fact that the last morphism is an isomorphism, we have $H_1(N,Y;\mathbb{Z})=H^3(N;\mathbb{Z})=0$.
Since we already know $H_2(N;\mathbb{Z})$ and $H_3(N;\mathbb{Z})$, the Euler characteristic of $N$ implies
the rank of $H_1(N;\mathbb{Z})$ is $2g$.
On the other hand, since $H_2(N;\mathbb{Z})=\mathbb{Z}$ is of rank one, we have that $H_2(N,Y;\mathbb{Z})$ is of rank one.
In the long exact sequence
$$H_2(N;\mathbb{Z}) \to H_2(N,Y;\mathbb{Z}) \to H_1(Y;\mathbb{Z}) \to H_1(N;\mathbb{Z}) \to 0$$
the map
$f:H_2(N;\mathbb{Z}) \to H_2(N,Y;\mathbb{Z})$ is given by multiplication by $[S_N]^2=k^2(2g-2)$ on the free generators of $H_2(N;\mathbb{Z})$ and $H_2(N,Y;\mathbb{Z})$.
The cokernal of $f$ contributes  a $\mathbb{Z}/(k^2(2g-2))\mathbb{Z}$  to $H_1(Y;\mathbb{Z})$.
Since $H_1(Y;\mathbb{Z})=\mathbb{Z}^{2g} \oplus \mathbb{Z}/(2g-2)\mathbb{Z}$, the torsion $\mathbb{Z}/(2g-2)\mathbb{Z}$ comes purely from the cokernal of $f$ and $k=1$.
It also implies that  $H_2(N,Y;\mathbb{Z})=\mathbb{Z}$ and $H_1(N;\mathbb{Z})=\mathbb{Z}^{2g}$.
We have now established the integral homology type as well as the intersection form for any exact filling $N$ of $Y$ and hence finished the proof.
\end{proof}

\begin{proof}[Proof of Corollary \ref{c:infiniteFamily}]
Let $Y_g$ be the standard unit cotangent bundle of an orinetable surface of genus $g$.
We note that $Y_g$ admits a Weinstein filling (See Example 11.12 (2) of \cite{CiEl12})
and hence a Stein (See Theorem 13.5 of \cite{CiEl12}) filling.

However, $Y_g$ has a semi-filling for $g>1$, call it $W_g$, with disconnected contact boundaries (Theorem 1.1 in \cite{Mc91}).
We can cap off the other boundary of $W_g$ by caps with arbitrarily large $b^+_2$ (See e.g. \cite{EtHo02}).
Hence, after blowing down the exceptional spheres, $W_g$ and these various caps can be glued together to give minimal strong fillings of $Y_g$ with arbitrarily large $b_2^+$.
Thus, $Y_g$ is not of strong Betti finite type.
\end{proof}

\subsection{Calabi-Yau caps via Lefschetz fibration}

For practical use, we want to have families of Calabi-Yau cap examples with concrete descriptions in terms of Lefschetz fibration and open book decompositions.
We refer to \cite{GS99} and \cite{OS04} for the basics of Lefschetz fibrations, open books and their relations to Stein fillings and contact structures.
For constructions of caps, the readers can consult \cite{Ozb06}.

Throughout the whole paper, we use the following notation. Let $\Sigma_g^k$ be a compact connected oriented surface of genus $g$ with $k$ boundary components, and let $\textnormal{Map}(\Sigma_g^k)$ be the mapping class group of $\Sigma_g^k$, i.e. the set of isotopy classes of self-diffeomorphisms of $\Sigma_g^k$ which preserve orientations
and fix the boundary $\partial \Sigma_g^k$ pointwise. We put $\Sigma_g=\Sigma_g^0$.
We denote by $\delta_1,\delta_2,\dots, \delta_k$ the boundary parallel curves of $\Sigma_g^k$. For a curve $C$ in $\Sigma_g^k$, let us denote the positive (i.e.\ right handed) Dehn twist along $C$ by $t_{C}$.

We recall how to construct a closed Lefschetz fibration of genus $g$ over $S^2$ by attaching some smooth 4-manifold to a Lefschetz fibration over $D^2$ with bounded fiber (cf.\ \cite{GS99}). Let $X$ be a Lefschetz fibration over $D^2$ with fiber $\Sigma_g^k$ whose induced open book on the boundary $\partial X$ is $(\Sigma_g^k, t_{\delta_k}^{i_k}\circ t_{\delta_{k-1}}^{i_{k-1}}\circ\dots \circ t_{\delta_{1}}^{i_{1}})$ $(i_1,i_2,\dots, i_k\geq 1)$.
By attaching a 2-handle to each binding component of the open book with page framing $0$, we get a Lefschetz fibration $X'$ over $D^2$ with closed fiber $\Sigma_g$.
By gluing $\Sigma_g\times D^2$ to $X'$, we obtain a closed Lefschetz fibration $\widetilde{X}$ of genus $g$ over $S^2$.

In terms of handles, the last gluing corresponds to the following operation. We first attach a 2-handle to $X'$ along a meridian of the attaching circle of the
each 2-handle attached to $X$. The Seifert framings of these 2-handles are $-i_1, -i_2, \dots, -i_k$, respectively. We then attach $2g+k-1$ 3-handles and one 4-handle to the resulting 4-manifold.

Let $LF_{g,k,I}$ be the smooth 4-manifold $\widetilde{X}- \textnormal{int}\, X$ for $I=(i_1,i_2,\dots, i_k)$, and put $LF_{g,k}=LF_{g,k,(1,1,\dots,1)}$. This is the neighborhood of a regular fiber and pairwise disjoint sections of $\widetilde{X}$. In particular, $LF_{g,k,I}$ is a plumbing of $\Sigma_{g}\times D^2$ and $D^2$-bundles over $S^2$ with Euler numbers $-i_1, -i_2,\dots, -i_k$.

Let $(Y_{g,k,I},\xi_{g,k,I})$ be the contact 3-manifold supported by the open book decomposition
$(\Sigma_{g}^k, t_{\delta_k}^{i_k}\circ t_{\delta_{k-1}}^{i_{k-1}}\circ\dots \circ t_{\delta_{1}}^{i_{1}})$ of the boundary of $X$, and put
$(Y_{g,k},\xi_{g,k})=(Y_{g,k,(1,1,\dots,1)},\xi_{g,k,(1,1,\dots,1)})$. We see a symplectic structure on $LF_{g,k}$ compatible with this contact structure using Gay's cap.

\begin{lemma}[cf.\ Gay~\cite{Ga03}, \cite{Ga03c}]\label{lem: LF cap} For any $g\geq 0$ and $k\geq 1$, $LF_{g,k}$ admits a symplectic cap structure such that
the contact structure on the concave boundary is isomorphic to $(Y_{g,k},\xi_{g,k})$. Furthermore, the suface of genus $g$ and the spheres in the plumbing are
symplectic submanifolds.
\end{lemma}
\begin{proof}Gay (See \cite{Ga03}, \cite{Ga03c}) constructed a symplectic cap structure on the $D^2$-bundle over $\Sigma_g$ such that the boundary contact structure
is supported by the open book $(\Sigma_{g,k}, t_{\delta_k}\circ t_{\delta_{k-1}}\circ\dots \circ t_{\delta_{1}})$. Furthermore, it contains a symplectic surface of
genus $g$ with the self-intersection number $k$. From his construction, it is easy to see that Gay's cap (say $G_{g,k}$) is obtained from $Y_{g,k}\times [0,1]$ by
attaching a 2-handle to $Y_{g,k}\times \{1\}$ along each binding component of the open book with page framing $+1$ and then attaching 3-handles and a 4-handle to
the resulting boundary. By blowing up this cap at $k$ points, we obtain a plumbing of symplectic surface of genus $g$ with self-intersection $0$ and $k$ spheres
with self-intersection $-1$. By sliding each of the aforementioned  2-handles of $G_{g,k}$ over the $-1$-framed unknot corresponding to the 2-handle of each
$\overline{\mathbb{C}\mathbb{P}^2}$, we easily see that $G_{g,k}\#k\overline{\mathbb{C}\mathbb{P}^2}$ has the same handle decomposition as $LF_{g,k}$, since 3-
and 4-handles are attached uniquely.
\end{proof}

\begin{rmk}This proof tells that 2-handles of $LF_{g,k}$ and $G_{g,k}\#_k\overline{\mathbb{C}\mathbb{P}^2}$ are attached to $Y_{g,k}$ along the same framed link.
\end{rmk}

A simple family of examples of Calabi-Yau caps are the $D^2$-bundle over $\Sigma_{g}$ with Euler number $2g-2$ $(g\geq 2)$
which is  Gay's cap $G_{g, 2g-2}$, and the concave boundary contact structure is supported by the Stein filable open book $(\Sigma_g^{2g-2}, t_{\delta_{2g-1}}\circ\dots \circ t_{\delta_{1}})$. Indeed the adjunction formula immediately tells that this cap is Calabi-Yau, since $G_{g,2g-2}$ contains a symplectic surface of genus $g$ with the self-intersection number $2g-2$.

The next proposition gives us many examples of Calabi-Yau caps $P$ (with $b_1(P)\leq 1$).
One can construct many examples of Stein fillable open books satisfying the assumption of the proposition. For example, consider the chain relation in $\textnormal{Map}(\Sigma_2^2)$ (see Lemma 21 in \cite{Waj99}). This relation gives rise to a monodromy factorization of a genus-$2$ Lefschetz fibration on $K3\#2\overline{\mathbb{C}{P}^2}$ (cf.\ P.201 in \cite{Sa13} and \cite{GS99}). Modifying the factorization by the Hurwitz equivalence, one can construct various sets of curves in $\Sigma_2^2$ satisfying the assumption below.

\begin{prop}\label{LF: CY}Let $(Y,\xi)$ be a contact 3-manifold supported by an open book $(\Sigma_g^{k}, \varphi)$ with $g\geq 1$ and $k\geq 1$,
where $\varphi=t_{C_n}\circ t_{C_{n-1}}\circ \dots \circ t_{C_1}$ for some homotopically non-trivial simple closed curves $C_1,\dots, C_n$ in $\Sigma_g^k$. Suppose that there
exist (possibly empty set of) homotopically non-trivial simple closed curves $C_{n+1}, C_{n+2},\dots, C_m$ $(m\geq n)$ in $\Sigma_g^{k}$ satisfying the following conditions.
\begin{itemize}
 \item $t_{C_m}\circ t_{C_{m-1}}\circ \dots \circ t_{C_1}=t_{\delta_{k}}\circ t_{\delta_{k-1}}\circ\dots \circ t_{\delta_{1}}$.
 \item The genus $g$ Lefschetz fibration $Z$ over $S^2$ with monodromy factorization $t_{C_m}\circ t_{C_{m-1}}\circ \dots \circ t_{C_1}$ is diffeomorphic
 to a symplectic Calabi-Yau surface blown up at $k$ points. Here we regard $C_1,\dots, C_m$ as curves in $\Sigma_g$ through the natural inclusion $\Sigma_g^{k}\hookrightarrow \Sigma_g$.
\end{itemize}
Then $(Y,\xi)$ admits a Calabi-Yau cap whose fundamental group is the quotient group $\pi_1(\Sigma_g)/\langle C_{n+1}, C_{n+2}, \dots, C_m \rangle$, where $\langle C_{n+1}, C_{n+2}, \dots, C_ m\rangle$ is the normal closure of the subgroup generated by these curves (More precisely, a base point of $\Sigma$ is fixed, and these curves are isotoped so that they pass through the base point.).
\end{prop}
\begin{proof}

By the assumption, the genus $g$ Lefschetz fibration $Z$ over $S^2$ has  $k$ sections with self-intersection $-1$. Let $X$ denote the blow down
of $Z$ along these sections. Since $Z$ is diffeomorphic to the blow up of a symplectic Calabi-Yau surface, and sections are a symplectic submanifold of
some symplectic structure on $Z$, the uniqueness of minimal model shows that $c_1(X)$ of the symplectic structure on $X$ is torsion.

Here we construct a Calabi-Yau cap. Let $(W=Y\times [0,1], \omega)$ be a symplectization of $(Y,\xi)$. We note that $Y\times \{0\}$ and $Y\times \{1\}$ are concave and convex boundary of $W$, respectively. By the Legendrian realization principle and folding, we may assume that $C_{n+1}$ is a Legendrian knot in a page of the open book $(\Sigma_g^k,\varphi)$
of $Y\times \{1\}$ (see Section 8.6 of \cite{A14} and \cite{OS04AMS}). We then attach a symplectic 2-handle to $W$ along $C_{n+1}$ with the framing $-1$ relative to the page framing. The resulting convex boundary is supported by the open book $(\Sigma_g^k, t_{C_{n+1}}\circ \varphi)$. Repeating this process, we obtain a compact symplectic 4-manifold $W'$ such that the concave boundary is $(Y,\xi)$ and that the convex boundary $(Y',\xi')$ is supported by the open book $(\Sigma_g^k, t_{C_{m}}\circ \dots \circ t_{C_{n+1}}\circ \varphi)$.
By Lemma~\ref{lem: LF cap} and the assumption $t_{C_{m}}\circ \dots \circ t_{C_{n+1}}\circ \varphi=t_{\delta_{k}}\circ\dots \circ t_{\delta_{1}}$, we can glue the cap $LF_{g,k}$ to $W'$ so that the resulting symplectic manifold $W''$ is a cap of $(Y,\xi)$. By blowing down  $k$ sections with self-intersection $-1$ contained in $LF_{g,k}\subset W''$, we obtain a symplectic cap $P$. Since $c_1(X)$ is torsion, $c_1(P)$ is also torsion. Hence $P$ is a Calabi-Yau cap of $(Y,\xi)$. By the proof of Lemma~\ref{lem: LF cap} and this construction, it is easy to see that $W''$ is obtained from $LF_{g,k}$ by attaching 2-handles along each vanishing cycle
$C_i$ $(n+1\leq i\leq m)$ in $\Sigma_g^k \times \partial D^2(\subset \Sigma_g \times \partial D^2)$ with page framing $-1$. Therefore $\pi_1(W'')$ is isomorphic
to $\pi_1(\Sigma_g)/\langle C_{n+1}, \dots, C_m \rangle$. Since the blowing down preserves the fundamental group, $P$ also has the same fundamental group.
\end{proof}

\begin{rmk} We were informed by Baykur that Proposition \ref{LF: CY} and Proposition \ref{LF:uniruled cap} overlap
with recent results by Baykur, Monden and Van Horn-Morris in \cite{BaMoVHM14}, which were obtained earlier and from different point of view.
In \cite{BaMoVHM14}, they obtain the sharp upper bound of the length of possible factorization of the open book $(\Sigma_k^g, t_{\delta_k} \circ  t_{\delta_{k-1}}\circ \cdots \circ t_{\delta_1})$ for any $k \ge 2g-3$, which implies sharp Betti numbers bounds among Stein fillings that have Lefschetz fibration compatible with the open book.
In contrast, our result implies Betti number bounds (not necessarily sharp) on any exact (resp. minimal strong) fillings of the contact manifold supported by the open book when $k = 2g-2$ (resp. $k \ge 2g-1$).
Our result does not cover the case that $k=2g-3$ and it would be interesting to know whether one can bound the Betti numbers for any exact/Stein fillings in this case.
\end{rmk}

\subsection{Calabi-Yau caps as divisor caps}\label{ss:divisor}

A {\it symplectic divisor} $D=C_1 \cup \dots \cup C_k$ in a symplectic manifold $(X,\omega)$ is a collection of closed embedded symplectic surfaces $C_i$
such that every intersection between any two $C_i$'s are positive and transversal and no three of the distinct $C_i$'s intersect at a common point.
If every intersection is also $\omega$-orthogonal, then we call $D$ an $\omega$-orthogonal divisor.
To each symplectic divisor, one can associate its augmented graph.
The augmented graph $(\Gamma,g,s,a)$ (or simply denoted as $(\Gamma,a)$ if no confusion would arise) is a weighted finite graph with vertices representing the surfaces and each edge joining two vertices representing an intersection between the two surfaces corresponding to the two vertices.
Moreover, each vertex $v_i$ is weighted by its genus (a non-negative integer $g_i$), its self-intersection number (an integer $s_i$) and its symplectic area (a positive real number $a_i$).
The intersection matrix associated to $(\Gamma,g,s,a)$ is denoted by $Q_{\Gamma}$.

\begin{defn}
 Suppose $(\Gamma,a)$ is an augmented graph with $k$ vertices with $a=(a_1,\dots,a_k)$.
Then, we say that $(\Gamma,a)$ satisfies the positive (resp. negative) Gay-Stipsicz  criterion   if there exist $z \in (0,\infty)^k$ (resp $(-\infty,0]^k$) such that $Q_{\Gamma}z=a$.
In either case, we will simply say that $(\Gamma,a)$ satisfies the GS criterion.
\end{defn}

For a symplectic divisor $D$ in $(X,\omega)$, the associated augmented graph is denoted by $(\Gamma_D,a)$.
We say that $(D,\omega)$ has a concave (resp. convex) neighborhood $P(D)$ if $P(D)$ is a strong capping (resp. strong filling) of its boundary.
If in every neighborhood $N$ of $(D,\omega)$, there is a concave (resp. convex) neighborhood $P(D) \subset N$, then we call $(D,\omega)$ a concave (resp. convex) divisor.

We recall two Theorems in \cite{LiMa14} (see also \cite{GaSt09} and \cite{McL14}).

\begin{thm}\label{criterion}
  Let $D \subset (X,\omega)$ be an $\omega$-orthogonal symplectic divisor.
Then, $D \subset (X,\omega)$ is a concave (resp. convex) divisor
if and only if $(\Gamma_D,a)$ satisfies the positive (resp. negative) GS criterion.
\end{thm}

\begin{thm}\label{non-negative implies concave}
  Let $D \subset (X,\omega_0)$ be a symplectic divisor.
If the intersection form of $D$ is not negative definite and $\omega_0$ restricted to the boundary of plumbing of $D$ is exact,
then $\omega_0$ can be deformed through a family of symplectic form $\omega_t$ making $D$ a symplectic divisor for each $t$ and such that $(D,\omega_1)$ is a concave divisor.
Moroever, the contact structure constructed on the boundary of the concave neighborhood depends only on the graph of $D$.
\end{thm}

\begin{rmk}\label{non-degenerate implies exact}
 Suppose $(D,\omega)$ is a symplectic divisor.
 If the intersection form $Q_D$ of $D$ is non-degenerate, then $\omega$ is exact on the boundary (cf. Lemma 2.4 of \cite{LiMa14}).
 In particular, $D$ is concave or convex possibly after a symplectic deformation.
\end{rmk}

Having the Theorems above, to determine whether a neighborhood of a symplectic divisor gives a uniruled cap is a purely linear algebraic computation.

\begin{eg}\label{Example of Calabi-Yau Cap} Consider the following  symplectic divisor
with genus of the central vertex $v_1$ being $g$ and genera of the others being $0$:
$$    
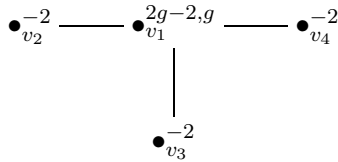
\xymatrix{
        \bullet^{-2}_{v_2} \ar@{-}[r]& \bullet^{2g-2,g}_{v_1} \ar@{-}[r] \ar@{-}[d]& \bullet^{-2}_{v_4} \\
			       & \bullet^{-2}_{v_3} \\
	}
$$
The $v_i$  have self-intersection numbers $2g-2$, $-2$, $-2$ and $-2$, respectively.
If $g > 1$, the intersection matrix is not negative definite and non-degenerate so there is a choice of symplectic form making it a concave divisor by Theorem \ref{non-negative implies concave}.
By the adjunction formula, it is easy to see that this concave neighborhood is a Calabi-Yau cap.
This configuration can be found in a rational manifold so the boundary of this cap is strongly fillable. In fact, it is not hard to show that this configuration is the support of an effective ample line bundle in an irrational ruled manifold and hence the contact structure is Stein fillable but we will not pursue further here.
\end{eg}

\section{Uniruled caps and adjunction caps}\label{s:UniAdj}

To motivate the definition of a uniruled cap, we recall a Theorem in \cite{Liu96} and \cite{OhOn96}.

\begin{thm}\label{uniruled closed}
 Let $(X,\omega)$ be a closed symplectic manifold with $c_1(X)\cdot[\omega] > 0$.
Then, $X$ is uniruled  (i.e.\ rational or ruled).
\end{thm}

Therefore, uniruled caps are the counterpart of uniruled manifolds for compact symplectic manifolds with boundary.


\begin{defn}
  An {\bf adjunction cap} of a contact 3-manifold $(Y,\xi)$ is a compact symplectic 4-manifold $(P,\omega)$ with strong concave boundary contactomorphic to $(Y,\xi)$
  such that there exist a smoothly embedded (not necessarily symplectic) genus $g$ surface $S$ in $P$ with
  $[S]^2 \ge \max \{2g-1,0\}$.
  If $[S]^2=0$, we further require that $[S] \in H_2(P,\mathbb{Q})$ does not lie in the image of $H_2(Y,\mathbb{Q})$ under the natural map induced by inclusion.
\end{defn}

The definition of adjunction caps is motivated by the following proposition
(Baykur informed us that he is aware of this statement but we can not find an explicit  reference so we present an argument here).

\begin{prop}\label{prop:adjunction}
Let $(X,\omega)$ be a closed symplectic manifold with a smoothly embedded (not necessarily symplectic) genus $g$ surface $S$ having self-intersection $[S]^2 \ge \max \{2g-1,0 \}$.
If $[S]^2=0$, we also assume that $[S]$ represents a non-trivial class in $H_2(X,\mathbb{Q})$.
Then $(X,\omega)$ is uniruled.
\end{prop}

\begin{proof}
If $S$ is a sphere with $[S]^2\geq 0$ and $[S]$ represents a non-trivial class in $H_2(X,\mathbb{Q})$,
then $X$ is uniruled, by Corollary 2 of \cite{Li99}.

Now suppose $g>0$ and $[S]^2 \geq 2g-1>0$. Then $S$ violates the adjunction inequality for $b_2^+\geq 2$ symplectic 4-manifolds in \cite{KrMr94} and \cite{MoSzTa96}.
Therefore we must have $b_2^+(X)=1$.
By the last paragraph in the proof of Theorem A in \cite{LiLi02}, there is a positive integer $n$ such that $n[S]$ is represented by an embedded connected $\omega'$-symplectic surface $C$ for some symplectic form $\omega'$. The genus $g_C$ of $C$ is given by
$$2g_C-2=(n[S])^2-c_1(X)([C])=n^2s-c_1(X) \cdot [C],$$
where $s=[S]^2$.
Notice  that $n[S]$ has another smooth representative $T$ given by taking $n-1$ perturbed copies of $S$ with pairwise positive distinct intersections and smoothing out the intersection points.
The genus of $T$ is given by
$$g_T=ng+\frac{n(n-1)s}{2}-(n-1).$$
Therefore, $$2g_T-2=2ng+n(n-1)s-2n=n^2s+n(2g-2-s)< n^2s.$$

Since $C$ is a symplectic surface we must have $g_C\leq g_T$. However,
if  $X$ is not uniruled, we have $c_1(X) \cdot [C] \leq 0 $ by \cite{Liu96} (otherwise, we can inflate the symplectic form in the direction of $PD([C])$ and get a contradiction, cf. Lemma \ref{first source} below).
Therefore, $g_T < g_C$.
This is a contradiction and
therefore $X$ must be uniruled.

\end{proof}

\begin{prop}\label{prop:uniruled}
 Let $(P,\omega_P)$ be a uniruled/adjunction cap of $(Y,\xi)$.
 If $(X,\omega)$ is a closed symplectic manifold obtained by gluing a strong symplectic filling $(N,\omega_N)$ of $(Y,\xi)$ with $(P,\omega_P)$ along $(Y, \xi)$,
 then $(X,\omega)$ is uniruled.
\end{prop}

\begin{proof}
 First assume $(P,\omega_P)$ is a uniruled cap.
 There is a Liouville contact form $\alpha_P$ making $c_1(P) \cdot [(\omega_P,\alpha_P)] >0$.
 For the strong symplectic filling $(N,\omega_N)$ of $(Y,\xi)$, we have a Liouville contact form $\alpha_N$ on $\partial N$ making
 $(\partial N, \ker (\alpha_N))$ contactomorphic to $(Y,\xi)$.
 By Lemma \ref{l:splittingCohomologyPairing}, when $t$ is taken to be sufficiently large, $c_1(X)\cdot[\omega_X] >0$
 and hence $X$ is uniruled, by Theorem \ref{uniruled closed}.

 Now, we assume $(P,\omega_P)$ is an adjunction cap.
Let $S$ be the smoothly embedded surface in $P$ such that $s=[S]^2 \ge \max \{2g-1,0 \}$.
If $S$ is a sphere and $[S]^2=0$, we have that $[S]$ represents a non-trivial class in $H_2(X,\mathbb{Q})$ by the Mayer-Vietoris sequence and the assumption that
$[S]\in H_2(P,\mathbb{Q})$ does not lie in the image of $H_2(Y,\mathbb{Q})$ under the natural map.
Hence, $X$ is uniruled, by Proposition \ref{prop:adjunction}.
\end{proof}

\subsection{Strong Betti finiteness and basic properties}\label{section: properties and genus invariants}

We are ready to prove Theorem \ref{t:uniruled} as well as its adjunction cap version.

\begin{thm}\label{t:refinedUA}(cf. Theorem \ref{t:uniruled})
 Any uniruled/adjunction contact manifold $(Y,\xi)$ is of strong Betti finite type.
\end{thm}

\begin{proof}

Let the glued symplectic manifold between a uniruled/adjunction cap $P$ of $Y$ and a strong symplectic filling $N$ of $Y$ be $(X,\omega)$ as before.
By Proposition \ref{prop:uniruled}, $(X,\omega)$ is uniruled.

Following the proof of Theorem \ref{t:CY}, we have $b_2^+(P) = 1$ and $b_2^+(N)=0$.
Moreover, there is a bound depending only on $P$ for base genus of $X$.

The next step is to give a bound depending only on $P$ for the number of exceptional curves in $X$.
Similar to the proof of Theorem \ref{t:CY}, we perturb $[(\omega_P,\alpha_P)]$ a little bit such that there is a positive number $c$ and
an embedded surface $S$ representing $cPD([(\omega_P,\alpha_P)])$ and $[S]^2 \ge g(S)-1$.
This time, an exceptional class in $X$ does not have to pair positively with $[S]$ a priori because $[S]$ is not Poincar\'e dual to multiple of $\omega_X$ in general.
However, by Corollary \ref{c:neck}, we know that any exceptional class indeed pairs positively with $[S]$.
After this point, the rest of the proof is the same as the proof of Theorem \ref{t:CY}.

\end{proof}

Uniruled/adjunction caps behave well with respect to strong cobordisms.

\begin{lemma}\label{stable under strong cobordism}
Let $(W,\omega_W)$ be a strong cobordism with negative end $(\partial_- W,\ker(\alpha_{W}^-))$ and positive end $(\partial_+ W,\ker(\alpha_W^+))$.
If $(\partial_+ W,\ker(\alpha_W^+))$ is uniruled/adjunction, then so is $(\partial_- W,\ker(\alpha_{W}^-))$.
\end{lemma}

 \begin{proof}
 We first assume $(\partial_+ W,\ker(\alpha_W^+))$ is uniruled.
Let $(P, \omega_P)$ be a uniruled cap for $(\partial_+ W,\ker(\alpha_W^+))$.
By an analogue of Lemma \ref{l:splittingCohomologyPairing},
$(W,\omega_W)$ and $(P,t\omega_P)$ can be glued symplectically by inserting part of the symplectization of $(\partial_+ W,\ker(\alpha_W))$ for  some  large $t$.
This glued symplectic manifold is a uniruled cap when $t$ is sufficiently large.

When $(\partial_+ W,\ker(\alpha_W^+))$ is adjunction, we take an adjunction cap $(P_+,\omega_{P_+})$ together with the surface $S$ inside.
The statement follows easily if $[S]^2 > 0$.
When $[S]^2 =0$, we define $P_-=W \cup_{\partial_+ W} P_+$ and we
want to show that $[S] \in H_2(P_-,\mathbb{Q})$ does not lie in the image of $H_2(\partial_- W,\mathbb{Q})$.
By the Mayer-Vietoris sequence and the definition of being an adjunction cap, the natural map
$$f:H_2(P_+,\mathbb{Q}) \to  H_2(P_+,\partial_+ W,\mathbb{Q})$$
satisfies $f([S]) \neq 0$.
Notice that, $f$ factors through
$$H_2(P_+,\mathbb{Q}) \xrightarrow{g} H_2(P_-,\mathbb{Q}) \xrightarrow{h} H_2(P_-,\partial_- W,\mathbb{Q}) \xrightarrow H_2(P_-,W,\mathbb{Q}) = H_2(P_+,\partial_+ W,\mathbb{Q})$$
In particular, $f([S]) \neq 0$ implies that the image of $g([S])$ under $h$ is non-trivial.
By the Mayer-Vietoris sequence again, $g([S])$ does not lie in the image of the natural map from $H_2(\partial_- W,\mathbb{Q})$ to $H_2(P_-,\mathbb{Q})$ and hence $P_-$ is an adjunction cap.
\end{proof}

The following lemma provides a common source of uniruled and adjunction caps.

\begin{lemma}\label{first source}
 Suppose $(P,\omega_P)$ is a cap for $(Y,\xi)$ with a closed embedded symplectic surface $S$ not  intersecting $\partial P$, $[S]^2 \ge 0$ and $c_1(P)[S] > 0$.
Then, after a symplectic deformation, $(P,\omega_P')$ is a uniruled and an adjunction cap for $(Y,\xi)$.
\end{lemma}

\begin{proof}
Since $S$ is smoothly embedded symplectic surface with non-negative self-intersection, we can do inflation (See \cite{LaMc96}, \cite{LiUs06}) along $S$ to deform the symplectic form.
This gives a family of symplectic form $\omega_t$ on $P$ such that $[\omega_t]=[\omega_P]+t\iota_*(PD[S])$, where $\iota_* : H^2(P,\partial P;\mathbb{R}) \to H^2(P;\mathbb{R})$
is the natural map and $PD$ denotes the Lefschetz dual.
Let $\alpha_P$ be a choice of Liouville contact form on $\partial P$ with respect to $\omega_P$.
Then, $\alpha_P$ is also a Liouville contact form on $\partial P$ with respect to $\omega_t$ since inflation is local.
Then, $c_1(P)\cdot[(\omega_t,\alpha_P)]=c_1(P)\cdot ([(\omega_P,\alpha_P)]+tPD[S]) >0$ for sufficiently large $t$.
Hence, we can find $\omega_P'$ such that $(P,\omega_P')$ is a uniruled cap for $(Y,\xi)$.

On the other hand, since $S$ is embedded and symplectic, we have $[S]^2+2-2g(S)=c_1(P)[S] \ge 1$ and $[(\omega_P,\alpha_P)] \cdot [S] > 0$.
Here $g(S)$ is the genus of $S$.
Hence $[S]^2 \ge \max \{2g-1,0\}$ and $[S] \in H_2(P,\partial P)$ is non-trivial.
Thus $[S] \in H_2(P)$ does not lie in the image of $H_2(\partial P)$ under the natural morphism, which implies that $(P,\omega_P)$ is an adjunction cap.

\end{proof}

Theorem~\ref{t:refinedUA} together with an argument of \cite{Et04} gives the following byproduct. The reader should compare this with Corollary 1.4 in Albers-Bramham-Wendl \cite{AlBrWe10}.

\begin{corr}\label{cor: semifilling}If a contact 3-manifold is uniruled/adjunction, then any strong semifilling of the contact manifold has connected boundary.
\end{corr}

\begin{proof}
Suppose, to the contrary, that $(Y,\xi)$ admits a semifilling $W$ with disconnected boundary.
Since every contact 3-manifold has a cap with arbitrarily large $b^+_2$ (see \cite{EtHo02}), by capping off the boundary components of $W$ apart from $(Y,\xi)$,
we can construct a strong filling $N$ of $(Y,\xi)$ with $b_2^+(N)$ as large as we want. This contradicts to Theorem~\ref{t:refinedUA}.
\end{proof}

\subsection{Genus invariants and complexity}\label{ss:genusBound}

We define  genus invariants for uniruled/adjunction caps, and the related complexity of  contact 3-manifolds, which will provide explicit bounds for
strong and Stein fillings.

\begin{defn}\label{adjunction: genus definition}
For a uniruled/adjunction cap $P$, its {\bf maximal} (resp. {\bf minimal}) {\bf base genus} $g_{\max}(P)$ (resp. $g_{\min}(P)$)  is the maximal (resp. minimal) base genus (i.e.\ $\frac{b_1(X)}{2}$)
of the closed uniruled manifolds $X$ obtained by gluing  strong fillings to the boundary of $P$.
The {\bf base genus difference} $g_{\Delta}(P)$ is defined to be $g_{\max}(P)-g_{\min}(P)$.
If there is no strong filling for the boundary of the cap, $g_{\max}(P)$, $g_{\min}(P)$ are defined to be $-\infty$ and $g_{\Delta}(P)$ is defined to be $0$.
The {\bf surface genus} $g_s(P)$ of $P$ is the minimal genus $g$ required to find a continuous map $u: \Sigma_g \to P$ such that $u_*[\Sigma_g]^2 >0$.
Here, $\Sigma_g$ is the closed oriented surface of genus $g$.
\end{defn}

Immediately, we have the following lemma.

\begin{lemma}\label{cor: preliminary inequality}
 For any uniruled/adjunction cap  $(P,\omega_P)$, $g_s(P)$ is finite and
 we have $g_s(P) \ge g_{\max}(P) \ge g_{\min}(P)$.
\end{lemma}

\begin{proof}
 The bounds  $ g_{\max}(P)\le g_s(P)<\infty$ of $g_s(P)$ is contained in the proof of Theorem \ref{t:refinedUA}.
\end{proof}

\begin{defn}
Given a strongly fillable uniruled/adjunction contact manifold $(Y,\xi)$, we define the {\bf complexity} $c_u(Y,\xi)$ of $(Y,\xi)$ to be the infimum of $g_{\max}(P)$
among all uniruled caps and all adjunction caps $(P,\omega_P)$ of $(Y,\xi)$.
If a uniruled/adjunction contact manifold $(Y,\xi)$ is not fillable, we define $c_u(Y,\xi)=0$.
\end{defn}

\begin{lemma}\label{l:gDelta}
 If $P$ and $P_2$ are two uniruled/adjunction caps of $(Y,\xi)$, then $g_{\Delta}(P)=g_{\Delta}(P_2)$.
Consequently, $c_u(Y,\xi) \ge g_{\Delta}(P)$.
\end{lemma}

\begin{proof}
 Without loss of generality, we assume $(Y,\xi)$ is fillable.
 Let $N_{\min}$ and $N_{\max}$ be two strong fillings of $(Y,\xi)$ such that the glued manifolds $P \cup_Y N_{\min}$ and $P \cup_Y N_{\max}$ realize $g_{\min}(P)$ and $g_{\max}(P)$, respectively.
Notice that $e+\sigma$ of any uniruled manifold only depends on its base genus.
By Novikov additivity, $(e+\sigma)(P \cup_Y N_{\min})=(e+\sigma)(P) + (e+\sigma)(N_{\min})$ and similarly for $P \cup_Y N_{\max}$.
As a result, the difference $(e+\sigma)(N_{\min})-(e+\sigma)(N_{\max})$ determines $g_{\Delta}(P)$ and hence also $g_{\Delta}(P_2)$.
It is now clear that $g_{\Delta}(P)=g_{\Delta}(P_2)$.
\end{proof}

We remark that Lemma \ref{l:gDelta} implies that $g_{\Delta}(P)$ is an invariant of $(Y,\xi)$ so we also denote it as $g_{\Delta}(Y,\xi)$
and call it the base genus difference.
Since $c_u(Y,\xi)=0$ for planar manifolds (cf. Lemma \ref{rmk: uniruled-planar}), we also have $g_{\Delta}(Y,\xi)=0$ and one can regard the base genus difference as a planar obstruction.
For example, the contact boundaries for the uniruled/adjunction caps that we construct later in Example \ref{ex:e+sign} and Lemma \ref{lemma:gamx > gmin} are non-planar.

Many previous known contact manifolds of Betti finite type are of complexity zero. Especially we have

\begin{lemma}\label{rmk: uniruled-planar}
Every strongly fillable planar contact 3-manifold $(Y,\xi)$ admits a uniruled cap $(P,\omega_P)$ with $g_s(P)=0$.
In particular, we have $c_u(Y,\xi)=g_{\max}(P)=g_{\min}(P)=0$.
\end{lemma}

\begin{proof}
In \cite{Et04}, it is observed that every planar open book has a cap with an embedded self-intersection $0$ symplectic sphere $S$ inside.
Since $S$ is symplectic,    we can apply  Lemma \ref{first source} to conclude that $(Y, \xi)$ has a uniruled cap.  Next  we show that $g_s(P)=0$ by constructing
an embedded  smooth sphere with positive self-intersection.
Notice that  the construction in \cite{Et04} gives another smoothly embedded sphere $S_2$  intersecting $S$ positively and transversally at one point.
Suppose $S_2$ has self-intersection $k$,  by choosing $n\geq |k|+1$ parallel copies of $S$ and resolving all the positive intersection points of them with  $S_2$, we get an embedded  sphere of positive self-intersection in the cap.

\end{proof}

\subsubsection{Explicit bounds for strong fillings}

We now give explicit Betti number bounds for fillings of uniruled/adjunction contact 3-manifolds in terms of genus invariants.

\begin{lemma}\label{cor:strong:bound}
Let $P$ be a uniruled/adjunction cap of a contact 3-manifold $(Y,\xi)$, and let $N$ be a strong filling of $(Y,\xi)$. Then the following hold. \\
$(1)$ $   \alpha(P)-4g_{\max}(P)  \leq  (e+\sigma)(N)  \leq   \alpha(P)-4g_{\min}(P)$, where $\alpha(P)=4-(e+\sigma)(P)$. \\
$(2)$ $ (b_1+b_3)(N) \leq  4g_{\max}(P)+2b_1(Y)-(b_1+b_3)(P)$.

\end{lemma}

\begin{proof}
We continue to use the notation from the proof of Theorem \ref{t:refinedUA}.
Since $X$ is uniruled, $(e+\sigma)(X)$ only depends on the base genus of $X$, and hence $(e+\sigma)(X) \in [-4g_{\max}(P)+4, -4g_{\min}(P)+4]$.
By Novikov additivity,
\begin{eqnarray*}
 &&(e+\sigma)(N) \\
 &=&(e+\sigma)(X)-(e+\sigma)(P) \\
 &\in& [-4g_{\max}(P)-(e+\sigma)(P)+4, -4g_{\min}(P)-(e+\sigma)(P)+4]
\end{eqnarray*}

Here note that $b_2^{0}(Z) \leq b_1(\partial Z)$ for any compact connected oriented 4-manifold $Z$ with boundary,
where $b_2^0(Z)$ denotes the maximal dimension of subspaces of $H_2(Z;\mathbb{R})$ whose intersection from is represented by the zero matrix.
(This can be seen from the long exact sequence for the pair $(Z, \partial Z)$ as follows. By the argument used in the solution of Exercise 5.3.13(f) in \cite{GS99},
it follows that $b_2^0(Z)$ is the rank of the kernel of the homomorphism $H_2(Z;\mathbb{R})\to H_2(Z,\partial Z;\mathbb{R})$.
The exact sequence thus gives $b_2^0(Z)\leq b_2(\partial Z)=b_1(\partial Z)$.)
Therefore, we have
\begin{eqnarray*}
 &&(b_1+b_3)(N) \\
 &=& -(e+\sigma)(N)+1+2b_2^+(N)+b_2^{0}(N)  \\
 &\le& 4g_{\max}(P)+(e+\sigma)(P)-4+1+b_1(Y) \\
 &\le& 4g_{\max}(P)+(1-(b_1+b_3)(P)+2b_2^+(P)+b_2^0(P))-3+b_1(Y)\\
 &\le& 4g_{\max}(P)+2b_1(Y)-(b_1+b_3)(P).
\end{eqnarray*}

\end{proof}

\begin{corr}\label{lem: strong extension of Wand}
Let $(Y,\xi)$ be a uniruled/adjunction contact 3-manifold.
Then $g_{\Delta}(Y,\xi)=0$ is equivalent to $e(N)+\sigma(N)$ being a constant for any strong fillings $N$ of $(Y,\xi)$.
In particular, if $c_u(Y,\xi)=0$ then $e(N)+\sigma(N)$ is a constant.
\end{corr}

\begin{proof}
It follows directly from Lemma \ref{l:gDelta} and Lemma \ref{cor:strong:bound}.
\end{proof}

Using Lemma \ref{rmk: uniruled-planar}   one can see that Corollary \ref{lem: strong extension of Wand} can be viewed as an extension of Wand's result
quoted as Corollary  \ref{cor:Wand}.


\subsubsection{Explicit bounds for Stein fillings}

For Stein fillings, we can also define the corresponding $g^{Stein}_{\max}$, $g^{Stein}_{\min}$ and $g^{Stein}_{\Delta}$, analogous to $g_{\max}$, $g_{\min}$ and $g_{\Delta}$.  Clearly,
$$  g^{Stein}_{\min}(P) \geq  g_{\min}(P), \quad     g^{Stein}_{\max}(P) \leq g_{\max}(P), \quad     g^{Stein}_{\Delta}(P) \leq  g_{\Delta}(P).$$

We obtain concrete bounds on $e+\sigma$ of Stein fillings, which might be helpful to classify Stein fillings.

\begin{lemma}\label{cor: stein : bounds} Suppose $(Y,\xi)$ admits a uniruled/adjunction cap $P$. Then for any Stein filling $N$ of $(Y,\xi)$, the following inequalities hold.
$$      g^{Stein}_{\max}(P) \le  \lfloor
\dfrac {b_1(P)}{2}\rfloor$$
$$\alpha(P)-4g^{Stein}_{\max}(P)  \leq (e+\sigma)(N)\leq \alpha(P)-4g^{Stein}_{\min}(P),$$
where $\alpha(P)=4- (e+\sigma)(P)$, and
 $\lfloor r \rfloor$ denotes the maximal integer less than or equal to $r$.
\end{lemma}
\begin{proof}Let $N$ be a Stein filling of $(Y,\xi)$, and let $X$ be the closed uniruled 4-manifold obtained by gluing $N$ and $P$ along $(Y,\xi)$. Since any Stein filling has a handle decomposition which consists of $0$-, $1$- and $2$-handles, and $b_1$ of any uniruled manifold is even, we obtain $b_1(X)\leq 2\lfloor \dfrac {b_1(P)}{2}\rfloor$. Therefore, the claim follows from the fact $(e+\sigma)(N)=(e+\sigma)(X)-(e+\sigma)(P)$.
\end{proof}

In particular, we have the following extension of Corollary \ref{cor:Wand}.

\begin{corr}\label{cor:stein:e+sigma} Suppose $(Y,\xi)$ has a uniruled/adjunction cap $P$.
Assume further $b_1(P)\leq 1$.
Then $e+\sigma$ of a Stein filling of $(Y,\xi)$ does not depend on the choice of the Stein filling.
\end{corr}

\begin{proof}
This clearly follows from Lemma \ref{cor: stein : bounds}.
\end{proof}

\begin{rmk}\label{rmk: alternative}Here we give an alternative proof of the Wand's result Corollary~\ref{cor:Wand}.
This proof can be extended to contact 3-manifolds $Y$ admitting adjunction cap $P$ with $b_1(P)\leq 1$. We note that any minimal strong filling of a planar contact 3-manifold is a Stein filling up to deformation (See \cite{We10}).

By \cite{Et04}, it is easy to see that any planar contact 3-manifold has a simply connected adjunction cap. Therefore, the closed symplectic 4-manifold obtained by gluing the cap to a given Stein filling is simply connected.
Since this manifold violates the Seiberg-Witten adjunction inequality with $b_2^+ \ge 2$ (see \cite{KrMr94}, \cite{MoSzTa96}), this closed manifold has $b_2^+=1$.
Therefore, $e+\sigma$ of the closed manifold is independent of the choice of the Stein filling. Since $e+\sigma$ of the closed manifold is the sum of those of the Stein filling and the cap, the desired claim follows.
\end{rmk}

\subsection{Construction}\label{Examples for Uniruled/Adjunction Caps}

We provide two general constructions for uniruled/adjunction caps.
The first one is equipped with explicit open book descriptions and the second one comes from symplectic divisors.
In particular, these examples provide many contact manifolds of strong Betti finite type.
All the uniruled caps constructed in this subsection are also adjunction caps, and vice versa.

\subsubsection{Complexity and  open books}

The following result provides us uniruled/adjunction caps with $b_1\leq 1$ associated to  a certain class of contact 3-manifolds with
an open book.
Notice that, for such contact 3-manifolds, $e+\sigma$ of their Stein fillings is constant by Corollary \ref{cor:stein:e+sigma}.

\begin{prop}\label{LF:uniruled cap} Let $(Y,\xi)$ be a contact 3-manifold supported by an open book $(\Sigma_g^k, \varphi)$ with $g\geq 1 $ and $k\geq 2g-1$. If there exist (possibly empty set of) homotopicaly non-trivial simple closed curves $C_1, C_2,\dots, C_n$ in $\Sigma_g^k$
such that $t_{C_n}\circ t_{C_{n-1}}\circ\dots \circ t_{C_{1}} \circ \varphi = t_{\delta_k}\circ t_{\delta_{k-1}}\circ\dots \circ t_{\delta_{1}}$ in $\textnormal{Map}(\Sigma_g^k)$,
then $(Y,\xi)$ admits a cap which is uniruled and adjunction. Furthermore, the fundamental group of the cap is the quotient group $\pi_1(\Sigma_g)/\langle C_1, C_2, \dots, C_n \rangle$. 
\end{prop}

\begin{proof}[Proof of Proposition~\ref{LF:uniruled cap}]
Similarly to Proposition~\ref{LF: CY}, we obtain a symplectic cap $W''$ of $(Y, \xi)$ which contains $LF_{g,k}$. Let $P$ be a cap obtained from $W''$ by blowing down $k$ sections with self-intersection $-1$ in $LF_{g,k}$.
Due to Lemma~\ref{lem: LF cap}, $P$ contains a symplectic surface of genus $g$ with self-intersection $k(> 2g-2)$.
The adjunction formula and Lemma~\ref{first source} thus implies that $P$ is both uniruled and adjunction. Similarly to Proposition~\ref{LF: CY}, we see that $\pi_1(P)$ is isomorphic to $\pi_1(\Sigma_g)/\langle C_1, C_2, \dots, C_n \rangle$.
\end{proof}

Note that the relations in $\textnormal{Map}(\Sigma_g^k)$ obtained by \cite{KoOz08}, \cite{Ona10} and \cite{Ta12} give us many concrete examples of Stein fillable contact structures satisfying the above assumption.
Furthermore, there are non-planar contact 3-manifolds satisfying the assumption (e.g.\ the Stein fillable contact structure on $T^3$. See \cite{VHM07}. cf.\ \cite{Y14})

Using our building blocks $LF_{g,k,I}$ and $(Y_{g,k,I},\xi_{g,k,I})$, we will further construct a uniruled/adjunction contact 3-manifold with positive complexity.
To find this example, we observe the following Lemma.

\begin{lemma}\label{lem:obvious filling}
For each $g\geq 0$ and $k\geq 1$, there exists a Stein filling $N_{g,k}$ of $(Y_{g,k}, \xi_{g,k})$ such that the glued symplectic 4-manifold
$N_{g,k}\cup LF_{g,k}$ is a closed uniruled 4-manifold with $b_1=2g$.
\end{lemma}
\begin{proof}Let $N_{g,k}$ be the Lefschetz fibration over $D^2$ with fiber $\Sigma_g^k$ whose monodromy factorization is
$t_{\delta_k}\circ \dots \circ t_{\delta_1}$. Then $N_{g,k}$ is a Stein filling of $(Y_{g,k}, \xi_{g,k})$ (see \cite{AO1}, \cite{A14}, \cite{OS04AMS}).
It is easy to see that the glued closed 4-manifold is smoothly diffeomorphic to the genus $g$ Lefschetz fibration over $S^2$ whose vanishing cycles are all trivial curves.
Therefore, the closed manifold is uniruled and has $b_1=2g$.
\end{proof}

We can now easily construct the desired example.
\begin{example}\label{ex:e+sign}There exists a uniruled and adjunction cap of a contact 3-manifold $(Y,\xi)$ satisfying the following: $(Y, \xi)$ admits
two Stein fillings whose $e+\sigma$ are not equal to each other.
In particular, $c_u(Y,\xi) > 0$ and $(Y,\xi)$ is non-planar.

We construct this example as follows. Let us recall a result of Tanaka~\cite{Ta12}. He gave a factorization of $t_{\delta_{4g+4}}\circ t_{\delta_{4g+3}}\circ \dots \circ t_{\delta_1}$ into a product of positive Dehn twists in $\textnormal{Map}(\Sigma_g^{4g+4})$ for each $g\geq 1$. Furthermore, he proved that the factorization in
$\textnormal{Map}(\Sigma_g)$ is a monodromy factorization of a genus $g$ Lefschetz fibration over $S^2$ whose total space is diffeomorphic to
$\mathbb{C}\mathbb{P}^2\#(4g+5)\overline{\mathbb{C}\mathbb{P}^2}$. Since the vanishing cycles of the Lefschetz fibration are non-separating in $\Sigma_g$, the curves appearing in Tanaka's factorization are all homotopically non-trivial in $\Sigma_g^{4g+4}$. By capping off boundary components,
we immediately get such a factorization of $t_{\delta_{k}}\circ t_{\delta_{k-1}}\circ \dots \circ t_{\delta_1}$ in $\textnormal{Map}(\Sigma_g^{k})$ for each $1\leq k\leq 4g+4$. Now let $X_{g,k}$ be the Lefschetz fibration over $D^2$ with fiber $\Sigma_g^k$ that has this monodromy factorization.
By these observations, we see that $X_{g,k}$ is a Stein filling of $(Y_{g,k}, \xi_{g,k})$ and that the glued symplectic 4-manifold $X_{g,k}\cup LF_{g,k}$ is diffeomorphic to $\mathbb{C}\mathbb{P}^2\#(4g+5)\overline{\mathbb{C}\mathbb{P}^2}$. Since $(e+\sigma)(LF_{g,k})=3-2g$ and $(e+\sigma)(\mathbb{C}\mathbb{P}^2\#(4g+5)\overline{\mathbb{C}\mathbb{P}^2})=4$, it follows $(e+\sigma)(X_{g,k})=1+2g$. On the other hand, direct computation shows that the Stein filling of $(Y_{g,k}, \xi_{g,k})$ in Lemma~\ref{lem:obvious filling} satisfies $e+\sigma=1-2g$. Since $(Y_{g,k}, \xi_{g,k})$ admits a uniruled/adjunction cap in the case $k\geq 2g-1$, the desired claim follows.
\end{example}


The last example below tells us that a cap which is symplectically embedded into a uniruled 4-manifold is not necessarily uniruled or adjunction.

\begin{example}\label{ex:non-uniruled cap}There exists a cap $P$ of a contact 3-manifold $(Y,\xi)$ satisfying the following: $(Y, \xi)$ admits two Stein fillings
$N_1$ and $N_2$ such that the glued symplectic closed 4-manifold $N_1\cup P$ (resp.\ $N_2\cup P$) is uniruled (resp.\ not uniruled).
Consequently, the cap $P$ is neither uniruled nor adjunction.

To construct such an example we use the following fact:  for each $g\geq 2$, $t_{\delta_1}\in \textnormal{Map}(\Sigma_g^1)$ has a factorization
into positive Dehn twists along non-separating (thus homologically non-trivial) curves such that the genus $g$ Lefschetz fibration over $S^2$ given by this
factorization satisfies $b_2^+>1$ (e.g.\ $M_{C_{II}}$ in Section 4 of \cite{En08}). Therefore, the factorization gives a Stein filling $X_{g,1}$ of
$(Y_{g,1}, \xi_{g,1})$ $(g\geq 2)$ such that the glued symplectic closed 4-manifold $X_{g,1}\cup LF_{g,1}$ is not uniruled, since $X_{g,1}\cup LF_{g,1}$ is diffeomorphic
to the above Lefschetz fibration. The claim thus follows from Lemma~\ref{lem:obvious filling}.
\end{example}

\subsubsection{Realizing pairs $(g_{\max}(P),g_{\min}(P))$ via plumbings}

It is natural to ask when a non-negative integer pair can be realized by $(g_{\max}(P),g_{\min}(P))$.
The first Chern class and the class of symplectic form are completely understood if the cap is a symplectic divisor cap,
from which we construct many examples realizing different pairs $(g_{\max}(P),g_{\min}(P))$.




We begin with uniruled divisor caps with  $g_{\max}(P)=g_{\min}(P)=0$.

\begin{lemma}\label{cor: T^3}
 $(T^3,\xi_{std})$ admits a uniruled divisor cap $(P,\omega_P)$ with $g_s(P)=g_{\Delta}(P)=0$ and hence $c_u(T^3,\xi_{std})=0$
\end{lemma}

\begin{proof}
 Notice that  $(T^3,\xi_{std})$ is the unit circle cotangent bundle (with respect to some metric) of $T^2$ with contact structure induced by the canonical Liouville flow in $T^*T^2$.
We can embed a Lagrangian torus in $\mathbb{CP}^2$ as a fibre of a toric moment map.
Then the complement of a Weinstein neighborhood of the Lagrangian toric fiber is a plumbing of three complex lines and hence a uniruled cap, by Lemma \ref{first source}.
\end{proof}

\begin{rmk}
 By Lemma \ref{rmk: uniruled-planar}, $c_u(Y,\xi)=0$ if $(Y,\xi)$ is planar and strongly fillable.   Notice that $c_u(T^3,\xi_{std})=0$ by Lemma \ref{cor: T^3}.
However, $(T^3,\xi_{std})$ is not planar  since  any Stein filling of a planar contact 3-manifold is negative definite by \cite{Et04}  but   there is a Stein filling of  $(T^3,\xi_{std})$ which is not negative definite, namely, the disk cotangent bundle.
Therefore, the class of contact manifolds having $c_u=0$ is strictly larger than that of those having supporting genus  zero.
\end{rmk}

\begin{eg}\label{second source}
This example is of fundamental importance in \cite{OhOn05} and we want to show that it is uniruled of complexity zero.
Consider   an augmented graph as in Subsection \ref{ss:divisor} with the numbers above the vertices being the self-intersection numbers:
$$    
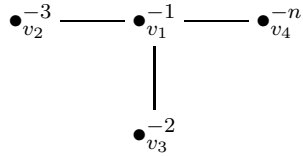
\xymatrix{
        \bullet^{-3}_{v_2} \ar@{-}[r]& \bullet^{-1}_{v_1} \ar@{-}[r] \ar@{-}[d]& \bullet^{-n}_{v_4} \\
			       & \bullet^{-2}_{v_3} \\
	}
$$
The genera are all zero and the symplectic areas are not specified.
In particular, the boundary $\partial P$ of the plumbing $P$  is a rational homology sphere when $n \neq 6$.
Therefore, we can identify the  second homology and cohomology group with $\mathbb{Q}$-coefficient
via the long exact sequence for the pair $(P, \partial P)$ and Lefschetz duality.
Thus
the first Chern class can be expressed as $PD(c_1)=\sum_{i=1}^4  a_i [v_i]$, where the $[v_i]$ are the homology classes of the corresponding spheres labeled in the graph and $PD$ stands for the Lefschetz dual. The coefficients $a_i$ are
calculated using the adjunction formula: $(\sum_{j=1}^4 a_j[v_j])\cdot [v_i]=[v_i]^2+2$ for each $i$.
It follows that $c_1=PD(2[v_1]+[v_2]+[v_3]+[v_4])$.

When $n=3,4,5$, the intersection form associated to the symplectic divisor is not negative definite and non-degenerate.
Therefore, Theorem \ref{non-negative implies concave} implies this symplectic divisor is concave for some choice of $\omega$.
Moreover, $c_1\cdot[\omega] > 0$ and hence it gives a uniruled cap.

If we choose an almost complex structure $J$ making the configuration $J$-holomorphic,
then we get a $(2,3)$-cusp curve with self-intersection $6-n$ by repeating blowing down.
This is the image of a sphere so we get $g_s(P)=0$.
\end{eg}

\begin{rmk}\label{rmk: unifying picture}
We have already seen that  how uniruled cap fits in the picture of some known Betti finite type contact manifolds.
In \cite{Et04}, \cite{Lis08}, \cite{PlVHM10}, \cite{KaL}, \cite{Ka13}, \cite{St13} and \cite{LiMa14}, those Betti finite type manifolds
have symplectic caps with a non-negative self-intersection symplectic sphere.
These are uniruled caps after a symplectic deformation, by Lemma \ref{first source}.
In \cite{OhOn05}, Ohta and Ono give some Betti finite type contact manifolds, which is the boundary of the ones in Example \ref{second source}.
In \cite{We10}, Wendl shows that  there is a unique, up to symplectic deformation,    minimal strong filling of $(T^3,\xi_{std})$.
We have seen in Lemma \ref{cor: T^3} that $(T^3,\xi_{std})$ also admits a uniruled cap.
In \cite{OhOn03}, their cap has a symplectic genus $g$ surface ($g \ge 1$) with self-intersection
greater than $2g-2$, hence can be deformed to be a uniruled cap (Lemma \ref{first source}).
The caps used in \cite{BhOn} also belong to one of the above families.
\end{rmk}


We move on to uniruled divisor caps with  $g_{\max}(P)=g_{\min}(P)>0$.

\begin{lemma}\label{lemma: simple plumbing}
 Let $D \subset (X,\omega)$ be a symplectic divisor with a symplectic sphere $C_1$ of self-intersection $0$ and a symplectic genus $g$ surface $C_2$ intersecting $C_1$ once.
 Suppose $X$ is a closed symplectic manifold and $\omega$ is exact on the boundary of a plumbing.
 Then there exists a plumbing $(P(D),\omega)$ which is a uniruled cap with $g_{\max}(P(D))=g_{\min}(P(D))=g$, possibly after a symplectic deformation.
\end{lemma}

\begin{proof}
 By \cite{LaMc96}, $X$ is ruled and $C_1$ represent a fiber.
 Since $C_2$ intersect $C_1$ transversally and positively once (See the Definition of symplectic divisor above), the projection from $C_2$ to the base is degree $1$.
 Therefore, the  genus  of the base  is the same as that of $C_2$, which is $g$.
 Moreover, $D$ has a concave neighborhood $P(D)$ by Theorem \ref{non-negative implies concave} and Remark \ref{non-degenerate implies exact}, possibly after a symplectic deformation.
 Furthermore, $P(D)$ can be chosen to be uniruled by Lemma \ref{first source}, possibly after a symplectic deformation.
\end{proof}

It is easy to give examples satisfying the assumption in Lemma \ref{lemma: simple plumbing}.
We can take $D$ to be the union of a fiber and a section of an $S^2-$bundle.
A more general construction is taking $X$ to be blown up of an $S^2-$bundle and
 $D$ can be taken to be the proper transform (or total transform, or mixed) of the union of a fiber and a section.


Finally we construct uniruled divisor caps  with $g_{\max}(P)>g_{\min}(P)=0$.
A tubular neighborhood of a symplectic genus $g$ ($g>0$) surfaces with appropriate self-intersection greater than $2g-2$ give a family of examples (cf. Lemma \ref{first source} and Theorem \ref{non-negative implies concave}).
Another family of examples is given in the proof of the following lemma.

\begin{lemma}\label{lemma:gamx > gmin}
 For any integer $n \ge 0$ and $K \le 4$, there is a uniruled cap $(P_{n,K},\omega_{n,K})$ which can be symplectically embedded in a ruled manifold with base genus $j$ for all $0 \le j \le n$.
 Moreover, one has $g_{\max}(P_{n,K})=n$ and $g_{\min}(P_{n,K})=0$.
\end{lemma}

\begin{proof}
Let $D_{n,K}$ be a symplectic divisor consisting of a symplectic sphere $C_1$ of self-intersection $0$, and a symplectic surface $C_2^{n,K}$ of genus $2n$ and self-intersection $K$ which intersects $C_1$ positively and transversally at two distinct points.
By  Theorem \ref{non-negative implies concave} and Remark \ref{non-degenerate implies exact}, any embedding of $D_{n,K}$ in a symplectic manifold can be made a concave divisor after possibly deforming the symplectic form, so we will take $(P_{n,K},\omega_{n,K})$ to be a concave neighborhood of $D_{n,K}$.
Notice that $(P_{n,K},\omega_{n,K})$ can be made a uniruled cap (Lemma \ref{first source}).
We will identify $D_{n,K}$ as a symplectic divisor in a suitable blow-up of $(S^2 \times \Sigma_j,\omega_j)$, where $\omega_j$ is a product symplectic structure,
for all $0 \le j \le n$.


 Let $\overline{C_1}$ be $S^2 \times \{ p\}$ for some $p \in \Sigma_j$.
 We construct another symplectic surface $\overline{C_2^{n,K}}$ by resolving two disjoint copies of $\{ \text{point} \} \times \Sigma_j$ and $2(n-j)+1$ disjoint copies of $S^2 \times \{ \text{point}\}$ different from $\overline{C_1}$.
 The symplectic surface $\overline{C_2^{n,K}}$ has genus $2n$ and intersect $\overline{C_1}$ positively transversally twice with self-intersection $8(n-j)+4$.
We blow up $\overline{C_2^{n,K}}$ at $8(n-j)+4-K$ regular points away from $\overline{C_1}$ and denote the proper transforms of $\overline{C_1}$ and $\overline{C_2^{n,K}}$ as $C_1$ and $C_2^{n,K}$, respectively.
The union of $C_1$ and $C_2^{n,K}$ is the desired symplectic divisor $D_{n,K}$, so the corresponding uniruled cap  $(P_{n,K},\omega_{n,K})$ has been embedded in a ruled surface of genus $j$ as promised.


To show that $g_{\max}(P_{n,K})=n$, we recall that if $P_{n,K}$ is embedded in a closed symplectic manifold $(X,\omega)$, then $X$ is ruled and $C_1$ represents a fiber class (See \cite{LaMc96}).
Therefore, the fact that $[C_2^{n,K}] \cdot [C_1]=2$ implies the projection of $C_2^{n,K}$ to the base has degree $2$.
Hence, the base genus is less than $n+1$ by Riemann-Hurwitz's formula.
\end{proof}


\begin{thebibliography}{99}
\bibitem{A14}
S.~Akbulut.
\newblock 4-manifolds.
\newblock book in preparation, available at http://www.math.msu.edu/$\sim$akbulut/

\bibitem{AO1}
S.~Akbulut and B.~Ozbagci.
\newblock Lefschetz fibrations on compact Stein surfaces.
\newblock {\em Geom. Topol.} 5: 319--334, 2001.


\bibitem{AY14}
S.~Akbulut and K.~Yasui.
\newblock Infinitely many small exotic Stein fillings.
\newblock {\em J.\ Symplectic Geom.} 12 (2014),  no.\ 4, 673--684.

\bibitem{AkEtMaSm08}
A.~Akhmedov, J.~Etnyre, T.~Mark, and I.~Smith.
\newblock A note on Stein fillings of contact manifolds.
\newblock  {\em Math. Res. Letters} 15(6): 127-133, 2008.


\bibitem{AkhOz2}
A.~Akhmedov and B.~Ozbagci.
\newblock Exotic Stein fillings with arbitrary fundamental group
\newblock arXiv:1212.1743v1, 2014.

\bibitem{AlBrWe10}
P.~Albers, B.~Bramham, and C.~Wendl.
\newblock On nonseparating contact hypersurfaces in symplectic 4-manifolds.
\newblock {\em Algebr. Geom. Topol.}  10(2): 697-737, 2010.

\bibitem{Ba08}
S. Bauer.
\newblock Almost complex 4-manifolds with vanishing first Chern class.
\newblock {\em J. Differential Geom.}, 79(1):25-32, 2008.


\bibitem{BaVHM15}
R.I.~Baykur and J.~Van Horn-Morris.
\newblock Families of contact 3-manifolds with arbitrarily large Stein fillings, with an appendix by S. Lisi and C. Wendl
\newblock {\em Journal of Differential Geometry}, 101(3):423-465, 2015.

\bibitem{BaVHM1212}
R.I.~Baykur and J.~Van Horn-Morris.
\newblock Topological complexity of symplectic 4-manifolds and Stein fillings.
\newblock arXiv:1212.1699, 2012. (to appear in {\em J. Symp. Geom.})

\bibitem{BaMoVHM14}
R.I.~Baykur, N.~Monden and J.V.~Horn-Morris.
\newblock Positive factorizations of mapping classes.
\newblock   arXiv:1412.0352, 2014.

\bibitem {BhOn}
M. Bhupal and K. Ono,
\newblock  Symplectic fillings of quotient surface singularities.
\newblock  {\em Nagoya Math. J.} 204 (2011), 1-45

\bibitem{BottTu}
R. Bott and L. Tu,
\newblock Differential forms in algebraic topology.
\newblock {\em Graduate Texts in Mathematics}, 82. Springer-Verlag, New York-Berlin, 1982.

\bibitem{BEHWZ03}
F. Bourgeois, Y. Eliashberg, H. Hofer, K. Wysocki and E. Zehnder.
\newblock Compactness results in symplectic field thoery.
\newblock {\em Geom. Topo.}, 7:799-888, 2003.

\bibitem{Bo86}
S.~Boyer,
\newblock Simply-connected 4-manifolds with a given boundary.
\newblock {\em Trans. Amer. Math. Soc.} 298 no. 1: 331--357, 1986.

\bibitem{Cas78}
J.W.S.~Cassels.
\newblock Rational quadratic forms.
\newblock London Mathematical Society Monographs, 13. Academic Press, Inc. [Harcourt Brace Jovanovich, Publishers], London-New York, 1978

\bibitem{CiEl12}
K.~Cieliebak and Y.~Eliashberg.
\newblock From Stein to Weinstein and back: symplectic geometry of affine complex manifolds.
\newblock {\em Amer. Math. Soc.} 2012.

\bibitem{DKP13}
E.~Dalyan, M.~Korkmaz and M.~Pamuk.
\newblock Arbitrarily Long Factorizations in Mapping Class Groups,
\newblock {\em Int. Math. Res. Not.} 2015(19): 9400--9414, 2015.

\bibitem{En08}
H.~Endo.
\newblock A generalization of Chakiris' fibrations.
\newblock Groups of diffeomorphisms,  251-282, Adv. Stud. Pure Math., 52, Math. Soc. Japan, Tokyo, 2008


\bibitem{El90}
Y.~Eliashberg.
\newblock Filling by holomorphic discs and its applications.
\newblock {\em Geometry of low-dimensional manifolds}, 2 (Durham 1989), volume 151 of {\em London Math. Soc. Lecture Note Ser.} pages 45-67 Cambridge Univ. Press, Cambridge, 1990.

\bibitem{ElGiHo00}
Y. Eliashberg, A. Givental and H. Hofer.
\newblock Introduction to symplectic field theory.
\newblock {\em Geom. Funct. Anal.}, special volume, (2):560-673, 2000.


\bibitem{El91}
Y Eliashberg.
\newblock On symplectic manifolds with some contact properties.
\newblock {\em J. Diff. Geom.}, 33: 233-238, 1991.

\bibitem{Et98}
J.~Etnyre.
\newblock Symplectic convexity in low-dimensional topology.
\newblock {\em Topol. and its Appl.} 88: 3-25, 1998.

\bibitem{Et04}
J.~Etnyre
\newblock Planar open book decompositions and contact structures.
\newblock {\em Int. Math. Res. Not.} 79: 4255-4267, 2004.

\bibitem{EtHo02}
J.~Etnyre and K.~Honda.
\newblock On symplectic cobordisms.
\newblock {\em Math. Ann.} 323: 31-39, 2002.

\bibitem{FrMo}
R.~Friedman and J.~Morgan.
\newblock Smooth Four-Manifolds and Complex Surfaces.
\newblock Ergebnisse der Mathematik und ihrer Grenzgebiete (3) [Results
              in Mathematics and Related Areas (3)], 27, Springer-Verlag, Berlin, 1994.


\bibitem{Ga03}
D.T.~Gay.
\newblock Open books and configurations of symplectic surfaces.
\newblock {\em Alg. Geom. Top.} 3: 569-586, 2003.

\bibitem{Ga03c}
D.T.~Gay.
\newblock Correction to "Open books and configurations of symplectic surfaces".
\newblock {\em Alg. Geom. Top.} 3: 1275-1276, 2003.

\bibitem{GaSt09}
D.T.~Gay and A.I.~Stipsicz.
\newblock Symplectic surgeries and normal surface singularities.
\newblock {\em Algebr. Geom. Topol.}, 9(4): 2203-2223, 2009.

\bibitem{GoLi14}
M. Golla and P. Lisca.
\newblock On Stein fillings of contact torus bundles.
\newblock {\em Bull. Lond. Math. Soc.} 48(1): 19-37, 2016.

\bibitem{GS99}
R.E~Gompf and A.I.~Stipsicz.
\newblock 4-manifolds and Kirby calculus.
\newblock Graduate Studies in Mathematics, 20. American Mathematical Society, Providence, RI, 1999.

\bibitem{Hi00}
R. Hind.
\newblock Holomorphic filling of $RP^3$.
\newblock {\em Commun. Contemp. Math.} 2(3): 349-363, 2000.

\bibitem{Ka13}
A.~Kaloti
\newblock Stein fillings of planar open books.
\newblock arXiv:1311.0208, 2013.

\bibitem{KaL}
A.~Kaloti and Y.~Li.
\newblock Stein fillings of contact 3-manifolds obtained as Legendrian surgeries.
\newblock arXiv:1307.4726, 2013.

\bibitem{KoOz08}
M.~Korkmaz and B.~Ozbagci.
\newblock On sections of elliptic fibrations.
\newblock {\em Michigan Math. J.} 56(1): 77-87, 2008.

\bibitem{KrMr94}
P.B. Kronheimer  and T.S. Mrowka.
\newblock The genus of embedded surfaces in the projective plane. (English summery)
\newblock {\em Math. Res. Lett.} 1(6): 797-808, 1994.




\bibitem{LaMc96}
F.~Lalonde and D.~McDuff.
\newblock The classification of ruled symplectic 4-manifolds.
\newblock {\em Math. Res. Lett.} 3: 769-778, 1996.

\bibitem{Le97}
C. LeBrun.
\newblock On the notion of general type.
\newblock {\em Rend. Mat. Appl.} (7) 17 no. 3: 513-522, 1997.


\bibitem{LiLi02}
B-H. Li and T-J. Li.
\newblock Symplectic genus, minimal genus and diffeomorphisms.
\newblock {\em Asian J. Math.} 6(1): 123-144, 2002.

\bibitem{Li99}
T-J. Li.
\newblock Smoothly embedded spheres in symplectic 4-manifolds.
\newblock {\em Proc. Amer. Math. Soc.} 127(2): 609-613, 1999.

\bibitem{Li06s}
T-J. Li.
\newblock Symplectic 4-manifolds with Kodaira dimension zero.
\newblock {\em J. Diff. Geom.} 74(2): 321-352, 2006.

\bibitem{Li06}
T-J. Li
\newblock Quaternionic bundles and Betti numbers of symplectic 4-manifolds with Kodaira dimension zero
\newblock {\em Int. Math. Res. Not.} Art.ID 37385:1-28, 2006.

\bibitem{LiLiu}
T-J. Li and A. Liu.
\newblock   Symplectic structures on ruled surfaces and a generalized adjunction inequality
\newblock   {\em Math. Res. Lett. } 2 (1995) 453-471.

\bibitem{LiMa14}
T-J. Li and C.Y. Mak.
\newblock Symplectic Divisorial Capping in Dimension 4.
\newblock arXiv:1407.0564.


\bibitem{LiUs06}
T-J.~Li and M.~Usher.
\newblock Symplectic forms and surfaces of negative square.
\newblock {\em J. Symplectic Geom.} 4(1): 71-91, 2006

\bibitem{LiZh09}
T-J. Li and W. Zhang.
\newblock Additivity and relative Kodaira dimensions.
\newblock Geometric and Analysis, Vol II (Yau's 60th birthday conference), 103--135, Advanced Lectures in Mathematics, No. 18.

\bibitem{Liu96}
A. Liu.
\newblock Some new applications of general wall crossing formula, Gompf's conjecture and its applications.
\newblock {\em Math. Res. Lett.} 3(5): 569-585, 1996.


\bibitem{Lis08}
P.~Lisca.
\newblock On symplectic fillings of lens spaces.
\newblock {\em Trans. Amer. Math. Soc.} 360(2): 765-799, 2008.

\bibitem{LVhmW}
S. Lisi, J. Van Horn-Morris and C. Wendl.
\newblock On symplectic fillings of spinal open book decomposition.
\newblock In preparation.




\bibitem{Mc90}
D McDuff.
\newblock The structure of Rational and Ruled symplectic 4-manifolds.
\newblock {\em J. Amer. Math. Soc.} 3:679-712, 1990.


\bibitem{Mc91}
D.~McDuff.
\newblock Symplectic manifolds with contact type boundaries.
\newblock {\em  Invent. Math.} 103(3): 651-671, 1991.

\bibitem{McSa96}
D. McDuff, D. Salamon.
\newblock A survey of symplectic 4-manifolds with $b^+ = 1$.
\newblock {\em Turkish J. Math.} 20(1):47-60, 1996.

\bibitem{McL14}
M.~McLean.
\newblock Reeb orbits and the minimal discrepancy of an isolated singularity.
\newblock {\em Invent. Math.} 204(2): 505--594, 2016.



\bibitem{MoSzTa96}
J.W. Morgan, Z. Szab$\acute{\text{o}}$ and C.H. Taubes.
\newblock A product formula for the Seiberg-Witten invariants and the generalized Thom conjecture.
\newblock {\em J. Differential Geom.} 44(4): 706-788, 1996.


\bibitem{OhOn96}
H.~Ohta and K.~Ono.
\newblock   Notes on symplectic 4-manifolds with $b^+=1$,  II.
\newblock {\em Int. J. Math. } 7: 755-770, 1996.

\bibitem{OhOn99}
H.~Ohta and K.~Ono.
\newblock       Simple singularities and topology of symplectically filling 4-
manifolds.
\newblock {\em Comment. Math. Helv. } 74: 575-590, 1999.

\bibitem{OhOn03}
H.~Ohta and K.~Ono.
\newblock Symplectic fillings of the link of simple elliptic singularities.
\newblock {\em J. Reine. Angew. Math.} 565: 183-205, 2003.

\bibitem{OhOn05}
H.~Ohta and K.~Ono.
\newblock Simple singularities and symplectic fillings.
\newblock {\em J. Diff. Geom.} 69(1): 1-42, 2005.

\bibitem{OhOn08}
H.~Ohta and K.~Ono.
\newblock Example of isolated surface singularities whose links have infinitely many symplectic fillings
\newblock {\em J. Fixed Point Theory Appl.} 3: 51-56, 2008.

\bibitem{Ona10}
S.C.~Onaran.
\newblock On sections of genus two Lefschetz fibrations.
\newblock {\em Pacific J. Math.} 248(1): 203-216, 2010.

\bibitem{Ozb06}
B.~Ozbagci.
\newblock Embedding fillings of contact 3-manifolds.
\newblock {\em Expo. Math.}  24(2): 161-183, 2006.

\bibitem{OS04}
B.~Ozbagci and A.I.~Stipsicz.
\newblock Surgery on contact 3-manifolds and Stein surfaces.
\newblock Bolyai Society Mathematical Studies, 13. Springer-Verlag, Berlin; Janos Bolyai Mathematical Society, Budapest, 2004.

\bibitem{OS04AMS}
B.~Ozbagci and A.I.~Stipsicz.
\newblock Contact 3-manifolds with infinitely many Stein fillings.
\newblock {\em Proc. Amer. Math. Soc.} 132(5): 1549-1558, 2004.

\bibitem{P12}
O.~Plamenevskaya.
\newblock On Legendrian surgeries between lens spaces.
\newblock {\em J. Symplectic Geom.} 10(2): 165--181, 2012.

\bibitem{PlVHM10}
O.~Plamenevskaya and J.~Van Horn-Morris.
\newblock Planar open books, monodromy factorizations and symplectic fillings.
\newblock {\em Geom. Topol.} 14(4): 2077--2101, 2010.

\bibitem{Pol}
 L. Polterovich,
\newblock The surgery of Lagrange submanifolds
\newblock{\em  Geom. Funct. Anal.} 1 (1991), no. 2, 198-210.

\bibitem{Sa13}
Y.~Sato.
\newblock Canonical classes and the geography of nonminimal Lefschetz fibrations over $S^2$.
\newblock {\em Pacific J. Math.}  262(1): 191-226, 2013.

\bibitem{SiVHM15}
S. Sivek and J.~Van Horn-Morris.
\newblock Fillings of unit cotangent bundles
\newblock arXiv:1510.06736, 2015.



\bibitem{St13}
L.~Starkston.
\newblock Symplectic fillings of Seifert fibered spaces.
\newblock {\em Trans. Amer. Math. Soc.} 367(8): 5971--6016, 2015.

\bibitem{Sti02}
A.I.~Stipsicz.
\newblock Gauge theory and Stein fillings of certain 3-manifolds.
\newblock {\em Turkish J. Math.} 26(1): 115-130, 2002.

\bibitem{Sti03}
A.I.~Stipsicz
\newblock On the geography of Stein fillings of certain 3-manifolds.
\newblock {\em Michigan Math J.} 51(2): 327-337, 2003.

\bibitem{Ta12}S.~Tanaka
\newblock On sections of hyperelliptic Lefschetz fibrations.
\newblock {\em Algebr. Geom. Topol.} 12(4): 2259-2286, 2012.

\bibitem{Taubes}C.~Taubes
\newblock   $SW\Rightarrow Gr$: From the Seiberg-Witten equations to pseudo-holomorphic curves
\newblock   {\em   J. Amer. Math. Soc. } 9 (1996) 845-918.

\bibitem{VHM07}J.~Van Horn-Morris
\newblock Constructions of open book decompositions.
\newblock Thesis (Ph.D.), The University of Texas at Austin, 2007.










\bibitem{Waj99}
B.~Wajnryb
\newblock An elementary approach to the mapping class group of a surface.
\newblock {\em Geom. Topol.} 3: 405-466, 1999.

\bibitem{Wa12}
A.~Wand
\newblock Mapping class group relations, Stein fillings, and planar open book decompositions.
\newblock {\em J. Topol.} 5(1): 1-14, 2012.

\bibitem{We10}
C.~Wendl
\newblock Strongly fillable contact manifolds and J-holomorphic foliations.
\newblock {\em Duke Math. J.} 151(3): 337-384, 2010.


\bibitem{We13}
C.~Wendl
\newblock A hierarchy of local symplectic filling obstructions for contact 3-manifolds.
\newblock {\em Duke Math. J.} 162(12): 2197-2283, 2013.

\bibitem{We13-2}
C.~Wendl
\newblock Non-exact symplectic cobordisms between contact 3-manifolds.
\newblock {\em J. Differential Geom.} 95(1): 121-182, 2013.

\bibitem{We14}
C.~Wendl
\newblock Contact hypersurfaces in uniruled symplectic manifolds always separate.
\newblock {\em J. London Math. Soc.}  89(3): 832-852, 2014.


\bibitem{Y14}
K.~Yasui
\newblock Partial twists and exotic Stein fillings.
\newblock arXiv:1406.0050, 2014

\end{thebibliography}
\end{document}